\documentclass[10pt]{article}

\usepackage{enumerate,xspace}
\usepackage{amsmath,amssymb,wasysym}
\usepackage[all]{xy}
\usepackage{proof}
\usepackage[svgnames]{xcolor}
\usepackage{pict2e} 
\usepackage{tikz}
\usepackage{stmaryrd} 
\usepackage{mathtools}
\usepackage{latexsym}
\usepackage{dsfont}
\usepackage{multicol}
\usepackage{cmll}
\usepackage{multirow}
\usepackage{longtable}
\usepackage[sort,nocompress]{cite} 
\usepackage{lscape}
\usepackage{array}

\usepackage{hyperref} 
\hypersetup{
    colorlinks,
    citecolor=red,
    filecolor=red,
    linkcolor=blue,
    urlcolor=red
}

\newtheorem{observation}{Remark}[section]
\newtheorem{lemma}[observation]{Lemma}  

\newtheorem{definition}[observation]{Definition}
\newtheorem{example}[observation]{Example}
\newtheorem{remark}[observation]{Remark}

\newtheorem{proposition}[observation]{Proposition} 
\newtheorem{corollary}[observation]{Corollary}

\makeatletter

\newdimen\w@dth

\def\setw@dth#1#2{\setbox\z@\hbox{\scriptsize $#1$}\w@dth=\wd\z@
\setbox\@ne\hbox{\scriptsize $#2$}\ifnum\w@dth<\wd\@ne \w@dth=\wd\@ne \fi
\advance\w@dth by 1.2em}

\def\t@^#1_#2{\allowbreak\def\n@one{#1}\def\n@two{#2}\mathrel
{\setw@dth{#1}{#2}
\mathop{\hbox to \w@dth{\rightarrowfill}}\limits
\ifx\n@one\empty\else ^{\box\z@}\fi
\ifx\n@two\empty\else _{\box\@ne}\fi}}
\def\t@@^#1{\@ifnextchar_ {\t@^{#1}}{\t@^{#1}_{}}}

\def\t@left^#1_#2{\def\n@one{#1}\def\n@two{#2}\mathrel{\setw@dth{#1}{#2}
\mathop{\hbox to \w@dth{\leftarrowfill}}\limits
\ifx\n@one\empty\else ^{\box\z@}\fi
\ifx\n@two\empty\else _{\box\@ne}\fi}}
\def\t@@left^#1{\@ifnextchar_ {\t@left^{#1}}{\t@left^{#1}_{}}}

\def\two@^#1_#2{\def\n@one{#1}\def\n@two{#2}\mathrel{\setw@dth{#1}{#2}
\mathop{\vcenter{\hbox to \w@dth{\rightarrowfill}\kern-1.7ex
                 \hbox to \w@dth{\rightarrowfill}}%
       }\limits
\ifx\n@one\empty\else ^{\box\z@}\fi
\ifx\n@two\empty\else _{\box\@ne}\fi}}
\def\tw@@^#1{\@ifnextchar_ {\two@^{#1}}{\two@^{#1}_{}}}

\def\tofr@^#1_#2{\def\n@one{#1}\def\n@two{#2}\mathrel{\setw@dth{#1}{#2}
\mathop{\vcenter{\hbox to \w@dth{\rightarrowfill}\kern-1.7ex
                 \hbox to \w@dth{\leftarrowfill}}%
       }\limits
\ifx\n@one\empty\else ^{\box\z@}\fi
\ifx\n@two\empty\else _{\box\@ne}\fi}}
\def\t@fr@^#1{\@ifnextchar_ {\tofr@^{#1}}{\tofr@^{#1}_{}}}

\newdimen\W@dth
\def\setW@dth#1#2{\setbox\z@\hbox{$#1$}\W@dth=\wd\z@
\setbox\@ne\hbox{$#2$}\ifnum\W@dth<\wd\@ne \W@dth=\wd\@ne \fi
\advance\W@dth by 1.2em}

\def\T@^#1_#2{\allowbreak\def\N@one{#1}\def\N@two{#2}\mathrel
{\setW@dth{#1}{#2}
\mathop{\hbox to \W@dth{\rightarrowfill}}\limits
\ifx\N@one\empty\else ^{\box\z@}\fi
\ifx\N@two\empty\else _{\box\@ne}\fi}}
\def\T@@^#1{\@ifnextchar_ {\T@^{#1}}{\T@^{#1}_{}}}

\def\T@left^#1_#2{\def\N@one{#1}\def\N@two{#2}\mathrel{\setW@dth{#1}{#2}
\mathop{\hbox to \W@dth{\leftarrowfill}}\limits
\ifx\N@one\empty\else ^{\box\z@}\fi
\ifx\N@two\empty\else _{\box\@ne}\fi}}
\def\T@@left^#1{\@ifnextchar_ {\T@left^{#1}}{\T@left^{#1}_{}}}

\def\Tofr@^#1_#2{\def\N@one{#1}\def\N@two{#2}\mathrel{\setW@dth{#1}{#2}
\mathop{\vcenter{\hbox to \W@dth{\rightarrowfill}\kern-1.7ex
                 \hbox to \W@dth{\leftarrowfill}}%
       }\limits
\ifx\N@one\empty\else ^{\box\z@}\fi
\ifx\N@two\empty\else _{\box\@ne}\fi}}
\def\T@fr@^#1{\@ifnextchar_ {\Tofr@^{#1}}{\Tofr@^{#1}_{}}}

\def\Two@^#1_#2{\def\N@one{#1}\def\N@two{#2}\mathrel{\setW@dth{#1}{#2}
\mathop{\vcenter{\hbox to \W@dth{\rightarrowfill}\kern-1.7ex
                 \hbox to \W@dth{\rightarrowfill}}%
       }\limits
\ifx\N@one\empty\else ^{\box\z@}\fi
\ifx\N@two\empty\else _{\box\@ne}\fi}}
\def\Tw@@^#1{\@ifnextchar_ {\Two@^{#1}}{\Two@^{#1}_{}}}

\def\to{\@ifnextchar^ {\t@@}{\t@@^{}}}
\def\from{\@ifnextchar^ {\t@@left}{\t@@left^{}}}
\def\tofro{\@ifnextchar^ {\t@fr@}{\t@fr@^{}}}
\def\To{\@ifnextchar^ {\T@@}{\T@@^{}}}
\def\From{\@ifnextchar^ {\T@@left}{\T@@left^{}}}
\def\Two{\@ifnextchar^ {\Tw@@}{\Tw@@^{}}}
\def\Tofro{\@ifnextchar^ {\T@fr@}{\T@fr@^{}}}

\makeatother

\title{Hopf Monads on Biproducts}
\author{Masahito Hasegawa and Jean-Simon Pacaud Lemay}

\begin{document}
\allowdisplaybreaks

\maketitle

\begin{abstract}
A Hopf monad, in the sense of Bruguières, Lack, and Virelizier, is a special kind of monad that can be defined for any monoidal category. In this note, we study Hopf monads in the case of a category with finite biproducts, seen as a symmetric monoidal category. We show that for biproducts, a Hopf monad is precisely characterized as a monad equipped with an extra natural transformation satisfying three axioms, which we call a fusion invertor. We will also consider three special cases: representable Hopf monads, idempotent Hopf monads, and when the category also has negatives. In these cases, the fusion invertor will always be of a specific form that can be defined for any monad. Thus in these cases, checking that a monad is a Hopf monad is reduced to checking one identity. 
\end{abstract}
\noindent \small \textbf{Acknowledgements.} For this research, 
the first named author was supported by
JSPS KAKENHI Grant No. JP18K11165 and JP21K11753.
The second named author was financially supported by a JSPS Postdoctoral Fellowship, Award \#: P21746. 



\section{Introduction}

Hopf monads were originally introduced by Bruguières and Virelizier in \cite{bruguieres2007hopf} for autonomous categories, and later generalized by Bruguières, Lack, and Virelizier in \cite{bruguieres2011hopf} to arbitrary monoidal categories. This latter definition of a Hopf monad requires that canonical natural transformations, called \emph{fusion operators}, are isomorphisms. The name ``Hopf monad" comes from the fact that a Hopf monad is a generalization of a Hopf algebra (whose antipode is invertible). Hopf monads have the ability to lift many desirable monoidal related structures. Indeed, for a monoidal closed category, Hopf monads are precisely the kinds of monads that lift the monoidal closed structure to their Eilenberg-Moore category \cite[Theorem 3.6]{bruguieres2011hopf}, and similarly for autonomous categories \cite[Theorem 3.10]{bruguieres2011hopf}. Furthermore, for star autonomous categories or traced monoidal categories, we can precisely characterize which Hopf monads lift star autonomous structure \cite{hasegawa2018linear} or traced monoidal structure \cite{hasegawa2022traced}. The purpose of this note is to discuss Hopf monads on a category with finite biproducts, viewed as a symmetric monoidal category. 

\begin{remark} It is important to point out and stress that the term ``Hopf monad" used in this paper and by Bruguières, Lack, and Virelizier in \cite{bruguieres2011hopf} differs from the term ``Hopf monad" used by Moerdijk in \cite{moerdijk2002monads}. Indeed, what Moerdijk calls a ``Hopf monad" in \cite{moerdijk2002monads} is what Bruguières et. al call a ``bimonad" in \cite{bruguieres2007hopf,bruguieres2011hopf}, which has also alternatively been called an ``opmonoidal monad" \cite{mccrudden2002opmonoidal} or a ``comonoidal monad" \cite{hasegawa2018linear}. So a ``Hopf monad" in the Bruguières et al. sense in \cite{bruguieres2011hopf} is a bimonad/opmonoidal monad/comonoidal monad whose canonical fusion operators are isomorphisms. Fusion operators and their inverses were not considered by Moerdijk in \cite{moerdijk2002monads}. As such, as this a follow-up on the work of Bruguières et al., we have elected to use the same terminology and definitions of ``bimonads" and ``Hopf monads" as they used in \cite{bruguieres2011hopf}. 
\end{remark}

For an arbitrary category $\mathbb{X}$, we denote a monad as a triple $(\mathsf{T}, \mu, \eta)$ where $\mathsf{T}: \mathbb{X} \to \mathbb{X}$ is the endofunctor, $\mu_A: \mathsf{T}\mathsf{T}(A) \to \mathsf{T}(A)$ is the multiplication, and $\eta_A: A \to \mathsf{T}(A)$ is the unit. For a monoidal category $\mathbb{X}$, with monoidal product $\otimes$ and monoidal unit $I$, a \textbf{bimonad} \cite[Section 2.4]{bruguieres2011hopf} is a monad $(\mathsf{T}, \mu, \eta)$ which comes equipped with a natural transformation ${\mathsf{m}_{A,B}: \mathsf{T}(A \otimes B) \to \mathsf{T}(A) \otimes \mathsf{T}(B)}$ and a map $\mathsf{m}_I: \mathsf{T}(I) \to I$ such that $\mathsf{T}$ is a comonoidal functor, and $\mu$ and $\eta$ are both comonoidal natural transformations. For any bimonad, there are two canonical natural transformations, which can always be defined, called the \textbf{fusion operators} \cite[Section 2.6]{bruguieres2011hopf}. The \textbf{left fusion operator} ${\mathsf{h}^l_{A,B}:  \mathsf{T}(A \otimes T (B)) \to \mathsf{T}(A) \otimes \mathsf{T}(B)}$ and the \textbf{right fusion operator} $\mathsf{h}^r_{A,B}:  \mathsf{T}(\mathsf{T}(A) \otimes B) \to \mathsf{T}(A) \otimes \mathsf{T}(B)$  are respectively defined as the following composites: 
\begin{align}
    \mathsf{h}^l_{A,B} &:=   \xymatrixcolsep{5pc}\xymatrix{ \mathsf{T}\left(A \otimes \mathsf{T}(B) \right) \ar[r]^-{\mathsf{m}_{A,\mathsf{T}(B)}} & \mathsf{T}(A) \otimes \mathsf{T}\mathsf{T}(B)\ar[r]^-{1_{\mathsf{T}(A)} \otimes \mu_B} & \mathsf{T}(A) \otimes \mathsf{T}(B) } \\
    \mathsf{h}^r_{A,B} &:=   \xymatrixcolsep{5pc}\xymatrix{\mathsf{T}\left(\mathsf{T}(A) \otimes B \right) \ar[r]^-{\mathsf{m}_{\mathsf{T}(A),B}} & \mathsf{T}\mathsf{T}(A) \otimes \mathsf{T}(B)\ar[r]^-{\mu_A \otimes 1_{\mathsf{T}(B)}} & \mathsf{T}(A) \otimes \mathsf{T}(B)}
\end{align}
A \textbf{Hopf monad} \cite[Section 2.7]{bruguieres2011hopf} is a bimonad whose fusion operators are natural isomorphisms, so we have that:
\[ \mathsf{T}(A \otimes \mathsf{T}(B)) \cong \mathsf{T}(A) \otimes \mathsf{T}(B) \cong \mathsf{T}(\mathsf{T}(A) \otimes B) \]
For a symmetric monoidal category, a \textbf{symmetric bimonad} \cite[Section 3, under the name ``cocommutative Hopf monad"]{moerdijk2002monads} is a bimonad that is also compatible with the natural symmetry isomorphism $\sigma_{A,B}: A \otimes B \to B \otimes A$. For a symmetric bimonad, the fusion operators can be defined from one another using the symmetry: $\mathsf{h}^r_{A,B} =  \sigma_{T(B), T(A)} \circ \mathsf{h}^l_{B,A} \circ T(\sigma_{\mathsf{T}(A), B})$. Therefore for a symmetric bimonad, the invertibility of one fusion operator implies the invertibility of the other \cite[Lemma 6.5]{hasegawa2022traced}. So for a symmetric bimonad, we may speak of \emph{the} fusion operator $\mathsf{h}_{A,B}:  \mathsf{T}(A \otimes T (B)) \to \mathsf{T}(A) \otimes \mathsf{T}(B)$, and define a \textbf{symmetric Hopf monad} \cite[Definition 6.4]{hasegawa2022traced} to be a symmetric bimonad whose fusion operator $\mathsf{h}$ is a natural isomorphism. The main examples of (symmetric) Hopf monads are those of the form $\mathsf{T}(-) = H \otimes -$ where $H$ is a (cocommutative) Hopf monoid with an invertible antipode \cite[Example 2.10]{bruguieres2011hopf}, and these Hopf monads are called \emph{representable} Hopf monads \cite[Section 5]{bruguieres2011hopf}. 

Before we discuss Hopf monads for biproducts, let us first discuss what one can say about Hopf monads for products and coproducts. Starting with products, let $\mathbb{X}$ be a category with finite products, with binary product $\times$, projections ${\pi_i: A_1 \times \hdots \times A_n \to A_i}$, pairing operator $\langle -, \hdots, - \rangle$, and terminal object $\ast$. Every monad $(\mathsf{T}, \mu, \eta)$ on a category $\mathbb{X}$ with finite products has a unique bimonad structure with respect to the Cartesian monoidal structure \cite[Example 1.6, under the name ``Hopf monad"]{moerdijk2002monads} where ${\mathsf{m}_{A,B}: \mathsf{T}(A \times B) \to \mathsf{T}(A) \times \mathsf{T}(B)}$ is defined as $\mathsf{m}_{A,B} := \left \langle \mathsf{T}(\pi_1), \mathsf{T}(\pi_2) \right \rangle$ and ${\mathsf{m}_\ast: \mathsf{T}(\ast) \to \ast}$ is defined as the unique map to the terminal object. In fact, this gives a symmetric bimonad structure. So the fusion operator $\mathsf{h}_{A,B}: \mathsf{T}(A \times T (B)) \to \mathsf{T}(A) \times \mathsf{T}(B)$ is worked out to be:
\[ \mathsf{h}_{A,B} = \left \langle \mathsf{T}(\pi_1), \mu_B \circ \mathsf{T}(\pi_2) \right \rangle \]
or in other words, using the universal property of the product, the fusion operator is the unique map which makes the following diagram commute: 
\[  \xymatrixcolsep{5pc}\xymatrix{ & \mathsf{T}\left( A \times \mathsf{T}(B) \right)  \ar@{-->}[dd]^-{\exists !~\mathsf{h}_{A,B}}  \ar[r]^-{\mathsf{T}(\pi_2)}  \ar@/_2pc/[ddl]_-{\mathsf{T}(\pi_1)}  & \mathsf{T}\mathsf{T}(B)  \ar[dd]^-{\mu_B} \\ \\
\mathsf{T}(A) & \ar[l]^-{\pi_1} \mathsf{T}(A) \times \mathsf{T}(B)  \ar[r]_-{\pi_2} & \mathsf{T}(B) 
}  \]
So for a category with finite products, we can say that a Hopf monad is a monad $(\mathsf{T}, \mu, \eta)$ such that the fusion operator $\mathsf{h}$ as defined above is a natural isomorphism. However, not much can be necessarily said about the form of the inverse of the fusion operator ${\mathsf{h}^{-1}_{A,B}: \mathsf{T}(A) \times \mathsf{T}(B) \to \mathsf{T}\left( A \times \mathsf{T}(B) \right)}$. 

On the other hand, what about the coproduct case? So now let $\mathbb{X}$ be a category with finite coproducts, with binary coproduct $\oplus$, injections $\iota_j: A_j \to A_1 \oplus \hdots \oplus A_n$, copairing operator $[-, \hdots, -]$, and initial object $\mathsf{0}$. Not every monad $(\mathsf{T}, \mu, \eta)$ on a category $\mathbb{X}$ with finite coproducts will have a bimonad structure with respect to the coCartesian monoidal structure. So we must ask that our monad $(\mathsf{T}, \mu, \eta)$ be a symmetric bimonad with structure maps ${\mathsf{m}_{A,B}: \mathsf{T}(A \oplus B) \to \mathsf{T}(A) \oplus \mathsf{T}(B)}$ and ${\mathsf{m}_\mathsf{0}: \mathsf{T}(\mathsf{0}) \to \mathsf{0}}$. In this case, not much can be said about the fusion operator $\mathsf{h}_{A,B}: \mathsf{T}(A \oplus T (B)) \to \mathsf{T}(A) \oplus \mathsf{T}(B)$. Instead, if we have a symmetric Hopf monad on a category with finite coproducts, we can say something about the inverse of the fusion operator since it is of type ${\mathsf{h}^{-1}_{A,B}: \mathsf{T}(A) \oplus \mathsf{T}(B) \to \mathsf{T}\left( A \oplus \mathsf{T}(B) \right)}$. Since $\mathsf{h}^{-1}_{A,B} \circ (\eta_A \oplus 1_{\mathsf{T}(B)}) = \eta_{A \oplus \mathsf{T}(B)}$ \cite[Lemma 4.2]{hasegawa2018linear}, by pre-composing by the injection $\iota_2$ and then using the naturality of $\eta$, it follows that: $\mathsf{h}^{-1}_{A,B} \circ \iota_2 = \mathsf{T}(\iota_2) \circ \eta_B$. In other words, by using the couniversal property of the coproduct, the inverse of the fusion operator is of the form:
\[ \mathsf{h}^{-1}_{A,B} := [ \mathsf{h}^\star_A, \mathsf{T}(\iota_2) \circ \eta_B  ] \]
for some unique map $\mathsf{h}^\star_A: \mathsf{T}(A) \to \mathsf{T}\left( A \oplus \mathsf{T}(B) \right)$, that is, $\mathsf{h}^{-1}_{A,B}$ is the unique map which makes the following diagram commute: 
\[  \xymatrixcolsep{5pc}\xymatrix{ \mathsf{T}(A)\ar[r]^-{\iota_1} \ar@/_2pc/[ddr]_-{\mathsf{h}^\star_{A,B}} &  \mathsf{T}(A) \oplus \mathsf{T}(B) \ar@{-->}[dd]_-{\exists!~ \mathsf{h}^{-1}_{A,B}} & \mathsf{T}(B) \ar[l]_-{\iota_2}   \ar[dd]^-{ \eta_{\mathsf{T}(B)}} \\ \\
 & \mathsf{T}\left( A \oplus \mathsf{T}(B) \right)  & \mathsf{T}\mathsf{T}(B)  \ar[l]^-{\mathsf{T}(\iota_2)} 
}  \]
Therefore, for coproducts, a symmetric Hopf monad can be characterized in terms of the existence of a natural transformation ${\mathsf{h}^\star_{A,B}: \mathsf{T}(A) \to \mathsf{T}\left( A \oplus \mathsf{T}(B) \right)}$ such that $\mathsf{h}_{A,B}\circ \mathsf{h}^\star_A = \iota_1$ and $[ \mathsf{h}^\star_A, \mathsf{T}(\iota_2) \circ \eta_B  ] \circ \mathsf{h}_{A,B} = 1_{\mathsf{T}\left( A \times \mathsf{T}(B) \right)}$. 

To recap, for products we have a full description of the fusion operator $\mathsf{h}$ but not the inverse $\mathsf{h}^{-1}$, while for coproducts we can't say much more on the form of the fusion operator $\mathsf{h}$ but know that the first argument of the inverse of the fusion operator $\mathsf{h}^{-1}$ must always be of a specific form. Since a biproduct is both a product and a coproduct, we can combine both observations to obtain a characterization of Hopf monads on categories with finite biproducts. Furthermore by naturality and using zero maps, it follows that $\mathsf{h}^\star_{A,B}: \mathsf{T}(A) \to \mathsf{T}\left( A \oplus \mathsf{T}(B) \right)$ is completely determined by the case where $B$  is the zero object $\mathsf{0}$. The main objective of this paper is to explain how a Hopf monad on a category with finite biproducts is precisely a monad $(\mathsf{T}, \mu, \eta)$ which comes equipped with an extra natural transformation $\mathsf{h}^\circ_A: \mathsf{T}(A) \to \mathsf{T}\left( A \oplus \mathsf{T}(\mathsf{0}) \right)$ satisfying three extra axioms, which we call a \textbf{fusion invertor}. We will also explain how in the cases of a representable Hopf monad, an idempotent Hopf monad, or in a setting with negatives, the fusion invertor is always of a specific form and how in these cases, checking that one has a Hopf monad is reduced to checking one identity. 

Lastly, before diving into this story, let us quickly discuss lifting again. A category with finite biproducts seen as a monoidal category is closed if and only if it is trivial, that is, all objects are zero objects. So using Hopf monads to lift (compact) closed structure or (star) autonomous structure in the biproduct setting is not interesting. On the other hand, it is possible to have a trace operator for biproducts. Per \cite[Corollary 7.10]{hasegawa2022traced}, for a traced coCartesian monoidal category, a Hopf monad lifts the trace operator if and only if the monad is an idempotent monad. As such, the same is true in the biproduct setting. At first glance, it may seem this clashes with another result which says that representable Hopf monads always lift trace operators \cite[Proposition 7.3]{hasegawa2022traced}. However, in a traced Cartesian monoidal category, the only Hopf monoid is the terminal object \cite[Theorem 3.7]{selinger2003order}, which means the only possible representable Hopf monad is the identity monad, which is trivially idempotent. Of course, for a category with finite biproducts, there can be non-representable idempotent Hopf monads, which in a traced setting would lift the trace operator. 

\section{Hopf Monads for Biproducts}\label{sec:hopfbiprod}

Let us reformulate the story of Hopf monads given in the introduction but this time described fully in terms of biproduct structure and its induced additive structure. For a category $\mathbb{X}$ with finite biproducts, we denote the biproduct of a family of $n$ objects as $A_1 \oplus \hdots \oplus A_n$ with injections $\iota_j: A_j \to A_1 \oplus \hdots \oplus A_n$ and projection $\pi_j: A_1 \oplus \hdots \oplus A_n \to A_j$, and denote the zero object as $\mathsf{0}$. For the induced commutative monoid enrichment, we denote the addition of parallel maps $f: A \to B$ and $g: A \to B$ as $f + g$ and the zero maps as $0: A \to B$. Lastly, recall that we denote a monad on a category $\mathbb{X}$ as a triple $(\mathsf{T}, \mu, \eta)$ where $\mathsf{T}: \mathbb{X} \to \mathbb{X}$ is the endofunctor, and $\mu_A: \mathsf{T}\mathsf{T}(A) \to \mathsf{T}(A)$ and $\eta_A: A \to \mathsf{T}(A)$ are the natural transformations.

As explained in the introduction, for a category with finite (bi)products, every monad always has a unique symmetric bimonad structure with respect to the (bi)product. Therefore, in the case of (bi)products we can define fusion operators and Hopf monads simply in terms of a monad. In the case of biproducts, the fusion operator can be defined in terms of a sum.   

\begin{definition} For a monad $(\mathsf{T}, \mu, \eta)$ on a category $\mathbb{X}$ with finite biproducts, the \textbf{fusion operator} is the natural transformation $\mathsf{h}_{A,B}: \mathsf{T}\left( A \oplus \mathsf{T}(B) \right) \to \mathsf{T}(A) \oplus \mathsf{T}(B)$ defined as follows: 
\begin{align}
\mathsf{h}_{A,B} = \iota_1 \circ \mathsf{T}(\pi_1) + \iota_2 \circ \mu_B \circ \mathsf{T}(\pi_2)
\end{align}
A \textbf{Hopf monad} on a category $\mathbb{X}$ with finite biproducts is a monad $(\mathsf{T}, \mu, \eta)$ on $\mathbb{X}$ whose fusion operator $\mathsf{h}$ is a natural isomorphism, so in particular $\mathsf{T}\left( A \oplus \mathsf{T}(B) \right) \cong \mathsf{T}(A) \oplus \mathsf{T}(B)$. 
\end{definition}

A list of identities that the fusion operator satisfies can be found in \cite[Proposition 2.6]{bruguieres2011hopf}, and a list of identities that the inverse of the fusion operator satisfies can be found in \cite[Lemma 4.2]{hasegawa2018linear}. Here are some examples of Hopf monads on biproducts. 

\begin{example}\label{ex:Gx-}\normalfont Let $\mathsf{CMON}$ be the category of commutative monoids (written additively) and monoid morphisms between them. $\mathsf{CMON}$ is a category with finite biproducts where the biproduct is given by the Cartesian product, $M \oplus N = M \times N$, with the monoid structure given pointwise, and where the zero object is the singleton $\mathsf{0} = \lbrace 0 \rbrace$. For any Abelian group $G$, $\mathsf{T}(-) := G \oplus -$ is a Hopf monad where the monad structure is given by:
\begin{align*} \mu_M: G \oplus M \oplus M \to M && \eta_M: M \to G \oplus M \\
\mu_M(g,h,m) = (g+h, m) && \eta_M(m) = (0,m) 
\end{align*}
and where the fusion operator and its inverse are: 
\begin{align*} \mathsf{h}_{M,N} : G \oplus M \oplus  G \oplus  N \to G \oplus  M \oplus  G \oplus  N && \mathsf{h}^{-1}_{M,N} : G \oplus  M \oplus  G \oplus  N \to G \oplus  M \oplus  G \oplus  N\\ 
\mathsf{h}_{M,N}(g,m,h,n) = (g, m, g+h, n) && \mathsf{h}^{-1}_{M,N}(g,m,h,n) = (g, m, h-g, n) 
\end{align*}
This is an example of a \emph{representable} Hopf monad, and can be generalized to any category with finite biproducts, which we discuss in Section \ref{sec:representable}. 
\end{example}

\begin{example}\label{ex:Z2} \normalfont Let $\mathsf{Ab}$ be the category of Abelian groups and group morphisms between them. $\mathsf{Ab}$ is a category with finite biproducts, with the same biproduct structure as $\mathsf{CMON}$. Let $\mathbb{Z}$ be the ring of integers and let $\mathbb{Z}_2$ be the ring of integers modulo $2$. Then $\mathsf{T}(G) := \mathbb{Z}_2 \otimes_\mathbb{Z} G$ (where $\otimes_\mathbb{Z}$ is the tensor product of Abelian groups/$\mathbb{Z}$-modules) is a Hopf monad where the monad structure is given by:
\begin{align*} \mu_G: \mathbb{Z}_2 \otimes_\mathbb{Z} \mathbb{Z}_2 \otimes_\mathbb{Z} G \to \mathbb{Z}_2 \otimes_\mathbb{Z} G && \eta_G: G \to \mathbb{Z}_2 \otimes_\mathbb{Z} G \\
\mu_G\left( x \otimes y \otimes g \right) = xy \otimes g && \eta_G(g) = 1 \otimes g 
\end{align*}
and where the fusion operator and its inverse are: 
\begin{align*} \mathsf{h}_{G,H} : \mathbb{Z}_2 \otimes_\mathbb{Z} \left( G \oplus \left( \mathbb{Z}_2 \otimes_\mathbb{Z} H \right) \right) \to \left( \mathbb{Z}_2 \otimes_\mathbb{Z} G \right) \oplus \left( \mathbb{Z}_2 \otimes_\mathbb{Z} H \right) && \mathsf{h}_{G,H}\left( x \otimes \left( g, y \otimes h  \right) \right) = \left( x \otimes g, xy \otimes h \right) \\ \mathsf{h}^{-1}_{M,N} :  \left( \mathbb{Z}_2 \otimes_\mathbb{Z} G \right) \oplus \left( \mathbb{Z}_2 \otimes_\mathbb{Z} H \right) \to \mathbb{Z}_2 \otimes_\mathbb{Z} \left( G \oplus \left( \mathbb{Z}_2 \otimes_\mathbb{Z} H \right) \right)  && \mathsf{h}^{-1}_{G,H}(x \otimes g, y \otimes h) = x \otimes \left( g, 0 \right) + y \otimes \left(0, y \otimes h \right) 
\end{align*}
We leave it as an exercise for the reader to check for themselves that $\mathsf{h}^{-1}_{G,H}$ is indeed well-defined and is the inverse of $\mathsf{h}_{G,H}$. The two reasons that $\mathbb{Z}_2 \otimes_\mathbb{Z} -$ is a Hopf monad follows from the fact that the biproduct $\oplus$ distributes over the tensor product $\otimes_\mathbb{Z}$, and that $\mathbb{Z}_2 \otimes_\mathbb{Z} \mathbb{Z}_2 \cong \mathbb{Z}_2$ via $x \otimes y \mapsto xy$ and $x \mapsto x \otimes x$. This latter fact actually implies that $\mathbb{Z}_2 \otimes_\mathbb{Z} -$ is an \emph{idempotent} Hopf monad, meaning in particular that $\mathsf{T}\mathsf{T}(G) \cong \mathsf{T}(G)$, which we will discuss more of in Section \ref{sec:idempotent}. It is worth noting that $\mathbb{Z}_2 \otimes_\mathbb{Z} -$ is also a (symmetric) Hopf monad with respect to the tensor product $\otimes_\mathbb{Z}$, and thus gives a (non-trivial) example of a monad which is a Hopf monad with respect to two different monoidal structures. However, we stress that while $\mathbb{Z}_2 \otimes_\mathbb{Z} -$ is a representable Hopf monad for the tensor product $\otimes_\mathbb{Z}$, it is not a representable Hopf monad for the biproduct $\oplus$ (since it is not of the form $G \oplus -$). 
\end{example}

\begin{example} \normalfont The above example generalizes to any monoidal category with finite biproducts such that the monoidal product $\otimes$ distributes over the biproduct $\oplus$. Then by taking any \emph{solid monoid} \cite[Definition 3.1]{gutierrez2015solid} $M$ with respect to the monoidal product $\otimes$ (i.e. a monoid whose multiplication is an isomorphism, so $M \otimes M \cong M$), $\mathsf{T}(-) = M \otimes -$ is a Hopf monad with respect to the biproduct $\oplus$. Other examples of solid monoids in $\mathsf{Ab}$ include $\mathbb{Z}$, the ring of rationals $\mathbb{Q}$, and the ring of $p$-adic integers for any prime $p$ (see \cite{gutierrez2015solid} for a full list of all solid monoids in $\mathsf{Ab}$). 
\end{example}

\begin{example}\normalfont \label{ex:HCxC} Here is a simple example that is neither representable nor idempotent. Consider the product category $\mathsf{CMON} \times \mathsf{CMON}$, which has a finite biproduct structure given pointwise. Then for any Abelian group $G$, $\mathsf{T}(M,N) := (G \oplus M, \mathsf{0})$ is a Hopf monad where the Hopf monad structure in the first argument is the same as in Example \ref{ex:Gx-}, while the Hopf monad structure in the second argument is trivially all zero maps. 
\end{example}

Let us now discuss what we can say about the form of the inverse of the fusion operator. For a Hopf monad on a category with finite biproducts, by the couniversal property of the coproduct, ${\mathsf{h}^{-1}_{A,B}: \mathsf{T}(A) \oplus \mathsf{T}(B) \to \mathsf{T}\left( A \oplus \mathsf{T}(B) \right)}$ is completely determined by its precomposition with the injection maps: $\mathsf{h}^{-1}_{A,B} \circ \iota_1: \mathsf{T}(A) \to \mathsf{T}\left( A \oplus \mathsf{T}(B) \right)$ and $\mathsf{h}^{-1}_{A,B} \circ \iota_2: \mathsf{T}(B) \to \mathsf{T}\left( A \oplus \mathsf{T}(B) \right)$. As explained in the introduction, for coproducts $\mathsf{h}^{-1}_{A,B} \circ \iota_2$ can always be shown to be $\mathsf{T}(\iota_1) \circ \eta_{\mathsf{T}(B)}$. Here we will provide an alternative proof using directly the biproduct structure. On the other hand, for coproducts, we cannot say much about $\mathsf{h}^{-1}_{A,B} \circ \iota_1$. For biproducts, however, we can show $\mathsf{h}^{-1}_{A,B} \circ \iota_1$ is completely independent of the $B$ term. 

\begin{lemma} \label{hinvi2}Let $(\mathsf{T}, \mu, \eta)$ be a Hopf monad on a category $\mathbb{X}$ with finite biproducts. Then for every pair of objects $A, B \in \mathbb{X}$, the following diagrams commute: 
    \[  \xymatrixcolsep{4pc}\xymatrix{\mathsf{T}(A) \ar[rr]^-{\iota_1} \ar[d]_-{\iota_1} && \mathsf{T}(A) \oplus \mathsf{T}(B)  \ar[d]^-{\mathsf{h}^{-1}_{A,B}} & \mathsf{T}(B) \ar[r]^-{\iota_2} \ar[d]_-{\eta_{\mathsf{T}(B)}} & \mathsf{T}(A) \oplus \mathsf{T}(B)  \ar[d]^-{\mathsf{h}^{-1}_{A,B}} \\
\mathsf{T}(A) \oplus \mathsf{T}(\mathsf{0}) \ar[r]_-{\mathsf{h}^{-1}_{A,\mathsf{0}}} & \mathsf{T}\left( A \oplus \mathsf{T}(\mathsf{0}) \right) \ar[r]_-{\mathsf{T}\left(1_A \oplus \mathsf{T}(0) \right)}  &  \mathsf{T}\left( A \oplus \mathsf{T}(B) \right) &  \mathsf{T}\mathsf{T}(B)  \ar[r]_-{\mathsf{T}(\iota_2)}  & \mathsf{T}\left( A \oplus \mathsf{T}(B) \right)
}  \]
\end{lemma}
\begin{proof} The diagram on the left follows from the naturality of the inverse of the fusion operator and the injection:
\begin{align*}
\mathsf{T}\left(1_A \oplus \mathsf{T}(0) \right) \circ \mathsf{h}^{-1}_{A,\mathsf{0}} \circ \iota_1 = \mathsf{h}^{-1}_{A,B} \circ \left( \mathsf{T}\left(1_A\right) \oplus \mathsf{T}(0) \right) \circ \iota_1 = \mathsf{h}^{-1}_{A,B} \circ \iota_1
\end{align*}
For the diagram on the right, we first compute the following: 
\begin{align*}
     \mathsf{h}_{A,B} \circ \mathsf{T}(\iota_2) \circ \eta_{\mathsf{T}(B)} &=~ \left( \iota_1 \circ \mathsf{T}(\pi_1) + \iota_2 \circ \mu_B \circ \mathsf{T}(\pi_2) \right) \circ \mathsf{T}(\iota_2) \circ \eta_{\mathsf{T}(B)} \tag{Def. of $\mathsf{h}$} \\
     &=~ \iota_1 \circ \mathsf{T}(\pi_1)\circ \mathsf{T}(\iota_2) \circ \eta_{\mathsf{T}(B)} + \iota_2 \circ \mu_B \circ \mathsf{T}(\pi_2)\circ \mathsf{T}(\iota_2) \circ \eta_{\mathsf{T}(B)} \\
     &=~ \iota_1 \circ \mathsf{T}(0) \circ \eta_{\mathsf{T}(B)} + \iota_2 \circ \mu_B \circ \eta_{\mathsf{T}(B)} \tag{$\pi_1 \circ \iota_2 =0$ and $\pi_2 \circ \iota_2 = 1$} \\
     &=~ \iota_1 \circ \eta_{A} \circ 0 + \iota_2 \tag{Nat. of $\eta$ and $\mu \circ \eta = 1$} \\
     &=~ 0 + \iota_2 \\
     &=~ \iota_2
\end{align*}
So $\mathsf{h}_{A,B} \circ \mathsf{T}(\iota_2) \circ \eta_{\mathsf{T}(B)} = \iota_2$. By post-composing by $\mathsf{h}^{-1}_{A,B}$, we obtain $\mathsf{T}(\iota_2) \circ \eta_{\mathsf{T}(B)} = \mathsf{h}^{-1}_{A,B} \circ \iota_2$, as desired. \hfill \end{proof}

By the above lemma, it follows that $\mathsf{h}^{-1}_{A,B}: \mathsf{T}(A) \oplus \mathsf{T}(B) \to \mathsf{T}\left( A \oplus \mathsf{T}(B) \right)$ must be of the form: 
\begin{align*}
\mathsf{h}^{-1}_{A,B} = \mathsf{T}\left(1_A \oplus \mathsf{T}(0) \right) \circ \mathsf{h}^{-1}_{A,\mathsf{0}} \circ \iota_1 \circ \pi_1 + \mathsf{T}(\iota_2) \circ \eta_{\mathsf{T}(B)} \circ \pi_2
\end{align*}
Looking more closely, every part on the right side can be defined for any monad except for $\mathsf{h}^{-1}_{A,\mathsf{0}}$, meaning that the only unfixed part of $\mathsf{h}^{-1}_{A,B}$ is completely determined by $\mathsf{h}^{-1}_{A,\mathsf{0}} \circ \iota_1$. Therefore it follows that a monad is a Hopf monad if and only if there exists a map of type ${\mathsf{h}^\circ_A: \mathsf{T}(A) \to \mathsf{T}\left( A \oplus \mathsf{T}(\mathsf{0}) \right)}$ satisfying the necessary identities such that:
\[\mathsf{h}^{-1}_{A,B} := \mathsf{T}\left(1_A \oplus \mathsf{T}(0) \right) \circ \mathsf{h}^\circ_A \circ \pi_1 + \mathsf{T}(\iota_2) \circ \eta_{\mathsf{T}(B)} \circ \pi_2 \]
is the inverse of the fusion operator $\mathsf{h}$. We call such a natural transformation $\mathsf{h}^\circ$ the \emph{fusion invertor}, since it is used in the construction of the inverse of the fusion operator, and define it below. 

\section{Fusion Invertor}\label{sec:fusioninvertor}

In this section, we introduce the main novel concept of this note, which is a fusion invertor, and prove our main result that for biproducts, a monad is a Hopf monad if and only if it has a fusion invertor. 

\begin{definition} A \textbf{fusion invertor} for a monad $(\mathsf{T}, \mu, \eta)$ on a category $\mathbb{X}$ with finite biproducts is a natural transformation $\mathsf{h}^\circ_A: \mathsf{T}(A) \to \mathsf{T}\left( A \oplus \mathsf{T}(\mathsf{0}) \right)$ such that: 
    \begin{enumerate}[{\bf [F{I}.1]}]
    \item $\mathsf{T}(\pi_1) \circ \mathsf{h}^\circ_A = 1_{\mathsf{T}(A)}$ 
    \item $\mu_\mathsf{0} \circ \mathsf{T}(\pi_2) \circ \mathsf{h}^\circ_A = 0$
    \item For every object $B$, $\mathsf{T}\left(1_A \oplus \mathsf{T}(0) \right) \circ \mathsf{h}^\circ_A \circ \mathsf{T}(\pi_1) + \mathsf{T}(\iota_2) \circ \eta_{\mathsf{T}(B)} \circ \mu_B \circ \mathsf{T}(\pi_2) =  1_{\mathsf{T}\left( A \oplus \mathsf{T}(B) \right)}$
    \end{enumerate}
\end{definition}

As we will see in the proof of Proposition \ref{mainprop} below, the three axioms of a fusion invertor are precisely what is required to construct an inverse for the fusion operator. Indeed, {\bf [F{I}.1]} and {\bf [F{I}.2]} are used to show that $\mathsf{h} \circ \mathsf{h}^{-1} = 1$, while {\bf [F{I}.3]} will be used to prove that $\mathsf{h}^{-1} \circ \mathsf{h} = 1$. A useful identity to have is the special case of {\bf [F{I}.3]} when $B=\mathsf{0}$ is the zero object. 

\begin{lemma} For a $(\mathsf{T}, \mu, \eta)$ monad with a fusion invertor $\mathsf{h}^\circ_A: \mathsf{T}(A) \to \mathsf{T}\left( A \oplus \mathsf{T}(\mathsf{0}) \right)$ on a category $\mathbb{X}$ with finite biproducts, for every object $A$, the following equality holds:
\begin{description}
\item[{\bf [F{I}.3.0]}] $\mathsf{h}^\circ_A \circ \mathsf{T}(\pi_1) + \mathsf{T}(\iota_2) \circ \eta_{\mathsf{T}(\mathsf{0})} \circ \mu_\mathsf{0} \circ \mathsf{T}(\pi_2) =  1_{\mathsf{T}\left( A \oplus \mathsf{T}(\mathsf{0}) \right)}$
\end{description}
\end{lemma}
\begin{proof} This immediately follows from the fact for the zero object $\mathsf{0}$, the identity map is equal to the zero map: $0 = 1_\mathsf{0}$. So {\bf [F{I}.3.0]} is simply rewritting {\bf [F{I}.3]} for the case $B=\mathsf{0}$ using that $\mathsf{T}\left(1_A \oplus \mathsf{T}(0) \right) = 1_{A \oplus \mathsf{T}(0)}$.   
\end{proof}

In particular, {\bf [F{I}.3.0]} will be very useful in showing that in special cases, the fusion invertor must always be of a specific form. We can also use {\bf [F{I}.3.0]} to show that a fusion invertor, if one exists, is unique, allowing one to speak of \emph{the} fusion operator. 

\begin{lemma} A fusion invertor, if it exists, is unique. \end{lemma} 
\begin{proof} Let $(\mathsf{T}, \mu, \eta)$ be a monad on a category $\mathbb{X}$ with finite biproducts. Suppose that $(\mathsf{T}, \mu, \eta)$ has two fusion invertors ${\mathsf{h}^\circ_A: \mathsf{T}(A) \to \mathsf{T}\left( A \oplus \mathsf{T}(\mathsf{0}) \right)}$ and $\mathsf{h}^\star_A: \mathsf{T}(A) \to \mathsf{T}\left( A \oplus \mathsf{T}(\mathsf{0}) \right)$. Then we compute: 
\begin{align*}
    \mathsf{h}^\star_A &=~ \left( \mathsf{h}^\circ_A \circ \mathsf{T}(\pi_1) + \mathsf{T}(\iota_2) \circ \eta_{\mathsf{T}(\mathsf{0})} \circ \mu_\mathsf{0} \circ \mathsf{T}(\pi_2) \right) \circ \mathsf{h}^\star_A \tag{{\bf [F{I}.3.0]}} \\
    &=~ \mathsf{h}^\circ_A \circ \mathsf{T}(\pi_1) \circ  \mathsf{h}^\star_A + \mathsf{T}(\iota_2) \circ \eta_{\mathsf{T}(\mathsf{0})} \circ \mu_\mathsf{0} \circ \mathsf{T}(\pi_2) \circ  \mathsf{h}^\star_A \\
    &=~ \mathsf{h}^\circ_A + \mathsf{T}(\iota_2) \circ \eta_{\mathsf{T}(\mathsf{0})} \circ 0 \tag{{\bf [F{I}.1]} and {\bf [F{I}.2]} for $\mathsf{h}^\star$} \\
    &=~ \mathsf{h}^\circ_A + 0 \\
    &=~ \mathsf{h}^\circ_A
\end{align*}
Thus $\mathsf{h}^\star_A = \mathsf{h}^\circ_A$. So we conclude that a fusion invertor, if it exists, is unique. 
\end{proof}

We now prove the main result of this paper. 

\begin{proposition}\label{mainprop} A monad on a category with finite biproducts is a Hopf monad if and only if it has a fusion invertor. Explicitly, if $(\mathsf{T}, \mu, \eta)$ is a monad on a category $\mathbb{X}$ with finite biproducts, then: 
    \begin{enumerate}[{\em (i)}]
\item If $(\mathsf{T}, \mu, \eta)$ is a Hopf monad, then $(\mathsf{T}, \mu, \eta)$ has a fusion invertor $\mathsf{h}^\circ_A: \mathsf{T}(A) \to \mathsf{T}\left( A \oplus \mathsf{T}(\mathsf{0}) \right)$ defined as the following composite: 
\begin{align}\label{hinvtohcirc}\mathsf{h}^\circ_A := \xymatrixcolsep{5pc}\xymatrix{ \mathsf{T}(A) \ar[r]^-{\iota_1} & \mathsf{T}(A) \oplus \mathsf{T}(\mathsf{0})  \ar[r]^-{\mathsf{h}^{-1}_{A,\mathsf{0}}} &  \mathsf{T}\left( A \oplus \mathsf{T}(\mathsf{0}) \right)
} 
\end{align} 
\item  If $(\mathsf{T}, \mu, \eta)$ has a fusion invertor $\mathsf{h}^\circ$, then $(\mathsf{T}, \mu, \eta)$ is Hopf monad where $\mathsf{h}^{-1}_{A,B}: \mathsf{T}(A) \oplus \mathsf{T}(B) \to \mathsf{T}\left( A \oplus \mathsf{T}(B) \right)$, the inverse of the fusion operator, is defined as follows:
\begin{align}\label{hcirctohinv}
\mathsf{h}^{-1}_{A,B} = \mathsf{T}\left(1_A \oplus \mathsf{T}(0) \right) \circ \mathsf{h}^\circ_A \circ \pi_1 + \mathsf{T}(\iota_2) \circ \eta_{\mathsf{T}(B)} \circ \pi_2
\end{align}
\end{enumerate}
\end{proposition}
\begin{proof} Suppose that $(\mathsf{T}, \mu, \eta)$ is a Hopf monad. By construction $\mathsf{h}^\circ$ is natural, so it remains to show that $\mathsf{h}^\circ$ satisfies the three identities \textbf{[FI.1]}, \textbf{[FI.2]}, and \textbf{[FI.3]}. So we compute: 
    \begin{enumerate}[{\em (i)}] \item $\mathsf{T}(\pi_1) \circ \mathsf{h}^\circ_A = 1_{\mathsf{T}(A)}$
\begin{align*}
\mathsf{T}(\pi_1) \circ \mathsf{h}^\circ_A &=~ \mathsf{T}(\pi_1) \circ  \mathsf{h}^{-1}_{A,B} \circ \iota_1 \tag{Def. of $\mathsf{h}^\circ$} \\
 &=~ \pi_1 \circ \mathsf{h}^\circ_A \circ \mathsf{h}^{-1}_{A,B} \circ \iota_1 \tag{Def. of $\mathsf{h}$} \\ 
&=~ \pi_1 \circ \iota_1 \\
&=~  1_{\mathsf{T}(A)} 
\end{align*}
\item $\mu_\mathsf{0} \circ \mathsf{T}(\pi_2) \circ \mathsf{h}^\circ_A = 0$
\begin{align*}
    \mu_\mathsf{0} \circ \mathsf{T}(\pi_2) \circ \mathsf{h}^\circ_A &=~ \mu_\mathsf{0} \circ \mathsf{T}(\pi_2) \circ \mathsf{h}^{-1}_{A,B} \circ \iota_1 \tag{Def. of $\mathsf{h}^\circ$} \\
    &=~ \pi_2 \circ \mathsf{h}^\circ_A \circ \mathsf{h}^{-1}_{A,\mathsf{0}} \circ \iota_1 \tag{Def. of $\mathsf{h}$} \\ 
    &=~ \pi_2 \circ \iota_1 \\
    &=~ 0 
\end{align*}
\item $\mathsf{T}\left(1_A \oplus \mathsf{T}(0) \right) \circ \mathsf{h}^\circ_A \circ \mathsf{T}(\pi_1) + \mathsf{T}(\iota_2) \circ \eta_{\mathsf{T}(B)} \circ \mu_B \circ \mathsf{T}(\pi_2) =  1_{\mathsf{T}\left( A \oplus \mathsf{T}(B) \right)}$
\begin{align*}
1_{\mathsf{T}\left( A \oplus \mathsf{T}(B) \right)} &=~ \mathsf{h}^{-1}_{A,B} \circ \mathsf{h}_{A,B} \\
&=~ \mathsf{h}^{-1}_{A,B} \circ \left( \iota_1 \circ \mathsf{T}(\pi_1) + \iota_2 \circ \mu_B \circ \mathsf{T}(\pi_2) \right) \tag{Def. $\mathsf{h}$} \\
&=~ \mathsf{h}^{-1}_{A,B} \circ \iota_1 \circ \mathsf{T}(\pi_1)   + \mathsf{h}^{-1}_{A,B} \circ \iota_2 \circ \mu_B \circ \mathsf{T}(\pi_2) \\
&=~ \mathsf{T}\left(1_A \oplus \mathsf{T}(0) \right) \circ  \mathsf{h}^{-1}_{A,\mathsf{0}} \circ \iota_1 \circ \mathsf{T}(\pi_1) +  \mathsf{T}(\iota_2) \circ \eta_{\mathsf{T}(B)} \circ \mu_B \circ \mathsf{T}(\pi_2) \tag{Lemma \ref{hinvi2}} \\
&=~ \mathsf{h}^\circ_A \circ \mathsf{T}(\pi_1) + \mathsf{T}(\iota_2) \circ \eta_{\mathsf{T}(B)} \circ \mu_B \circ \mathsf{T}(\pi_2) \tag{Def. of $\mathsf{h}^\circ$}
\end{align*}
\end{enumerate}
So we conclude that $\mathsf{h}^\circ$ is a fusion invertor. 

Conversely, suppose that $(\mathsf{T}, \mu, \eta)$ has a fusion invertor $\mathsf{h}^\circ$. We must show that $\mathsf{h}^{-1}$ is an inverse of $\mathsf{h}$. Using \textbf{[FI.3]}, we first compute that: 
 \begin{align*}
        \mathsf{h}^{-1}_{A,B} \circ \mathsf{h}_{A,B} &=~ \left(  \mathsf{T}\left(1_A \oplus \mathsf{T}(0) \right) \circ \mathsf{h}^\circ_A \circ \pi_1 + \mathsf{T}(\iota_2) \circ \eta_{\mathsf{T}(B)} \circ \pi_2 \right) \circ \mathsf{h}_{A,B} \tag{Def. of $\mathsf{h}^{-1}$} \\
        &=~  \mathsf{T}\left(1_A \oplus \mathsf{T}(0) \right) \circ \mathsf{h}^\circ_A \circ \pi_1 \circ \mathsf{h}_{A,B} + \mathsf{T}(\iota_2) \circ \eta_{\mathsf{T}(B)} \circ \pi_2 \circ \mathsf{h}_{A,B} \\
 &=~   \mathsf{T}\left(1_A \oplus \mathsf{T}(0) \right) \circ \mathsf{h}^\circ_A \circ \mathsf{T}(\pi_1) + \mathsf{T}(\iota_2) \circ \eta_{\mathsf{T}(B)} \circ \mu_B \circ \mathsf{T}(\pi_2) \tag{Def. of $\mathsf{h}$} \\
 &=~  1_{\mathsf{T}\left( A \oplus \mathsf{T}(B) \right)} \tag{\textbf{[FI.3]}}
    \end{align*}
So $\mathsf{h}^{-1}_{A,B} \circ \mathsf{h}_{A,B} =  1_{\mathsf{T}\left( A \oplus \mathsf{T}(B) \right)}$. For the other direction, first recall that in the proof of Lemma \ref{hinvi2}, we computed that:
\begin{align}
    \mathsf{h}_{A,B} \circ \mathsf{T}(\iota_2) \circ \eta_{\mathsf{T}(B)} = \iota_2 \label{hhinv1}
\end{align}
Using \textbf{[FI.1]} and \textbf{[FI.2]}, we can also compute that: 
\begin{align*}
     \mathsf{h}_{A,\mathsf{0}} \circ \mathsf{h}^\circ_A &=~ \left( \iota_1 \circ \mathsf{T}(\pi_1) + \iota_2 \circ \mu_\mathsf{0} \circ \mathsf{T}(\pi_2) \right) \circ \mathsf{h}^\circ_A \tag{Def. of $\mathsf{h}$} \\ 
     &=~ \iota_1 \circ \mathsf{T}(\pi_1)  \circ \mathsf{h}^\circ_A + \iota_2 \circ \mu_\mathsf{0} \circ \mathsf{T}(\pi_2)  \circ \mathsf{h}^\circ_A \\
     &=~ \iota_1 + \iota_2 \circ 0 \tag{\textbf{[FI.1]} and \textbf{[FI.2]}} \\
     &=~ \iota_1 + 0 \\
     &=~ \iota_1 
\end{align*}
So we have that: 
\begin{align}
   \mathsf{h}_{A,\mathsf{0}} \circ \mathsf{h}^\circ_A = \iota_1 \label{hhinv2}
\end{align}
Then we compute that: 
\begin{align*}
    \mathsf{h}_{A,B} \circ \mathsf{h}^{-1}_{A,B} &=~\mathsf{h}_{A,B} \circ  \left(  \mathsf{T}\left(1_A \oplus \mathsf{T}(0) \right) \circ \mathsf{h}^\circ_A \circ \pi_1 + \mathsf{T}(\iota_2) \circ \eta_{\mathsf{T}(B)} \circ \pi_2 \right) \tag{Def. of $\mathsf{h}^\circ$}  \\
    &=~  \mathsf{h}_{A,B}  \circ \mathsf{T}\left(1_A \oplus \mathsf{T}(0) \right) \circ \mathsf{h}^\circ_A \circ \pi_1 +  \mathsf{h}_{A,B} \circ \mathsf{T}(\iota_2) \circ \eta_{\mathsf{T}(B)} \circ \pi_2  \\
        &=~ \left( \mathsf{T}\left(1_A \right) \oplus \mathsf{T}(0) \right) \circ \mathsf{h}_{A,\mathsf{0}}  \circ \mathsf{h}^\circ_A \circ \pi_1 +  \mathsf{h}_{A,B} \circ \mathsf{T}(\iota_2) \circ \eta_{\mathsf{T}(B)} \circ \pi_2  \tag{Nat. of $\mathsf{h}$} \\
   &=~ \left( 1_{\mathsf{T}(A)} \oplus \mathsf{T}(0) \right) \circ \iota_1 \circ \pi_1 + \iota_2 \circ \pi_2 \tag{(\ref{hhinv1}) and (\ref{hhinv2})} \\
    &=~ \iota_1 \circ \pi_1 + \iota_2 \circ \pi_2 \tag{Nat. of $\iota_1$} \\
   &=~ 1_{\mathsf{T}(A) \oplus \mathsf{T}(B)} \tag{$\iota_1 \circ \pi_1 + \iota_2 \circ \pi_2 = 1$}
\end{align*}
So $\mathsf{h}_{A,B} \circ \mathsf{h}^{-1}_{A,B} = 1_{\mathsf{T}(A) \oplus \mathsf{T}(B)}$. So we conclude that the fusion operator is a natural isomorphism, and therefore that $(\mathsf{T}, \mu, \eta)$ is a Hopf monad. 
\end{proof}

We conclude this section by considering the fusion invertors for the Hopf monad examples in Section \ref{sec:hopfbiprod}. 

\begin{example}\normalfont \label{ex:Gfi} For any Abelian group $G$ and the induced Hopf monad $\mathsf{T}(-) := G \oplus -$ on $\mathsf{CMON}$ (Example \ref{ex:Gx-}), the fusion invertor $\mathsf{h}^\circ_M: G \oplus M \to G \oplus M \oplus G \oplus \mathsf{0}$ is defined as follows: 
\[ \mathsf{h}^\circ_M(g,m) = (g, m, -g, 0) \]
In Section \ref{sec:representable}, we will show that the fusion invertor is always of this form for any representable Hopf monad. 
\end{example}

\begin{example} \normalfont For the Hopf monad $\mathsf{T}(-):= \mathbb{Z}_2 \otimes_\mathbb{Z} -$ on $\mathsf{Ab}$ (Example \ref{ex:Z2}), the fusion invertor $\mathsf{h}^\circ_G: \mathbb{Z}_2 \otimes_\mathbb{Z} G \to \mathbb{Z}_2 \otimes_\mathbb{Z} \left( G \oplus \left( \mathbb{Z}_2 \otimes_\mathbb{Z} \mathsf{0} \right) \right)$ is defined as follows: 
    \[ \mathsf{h}^\circ_G(x \otimes g) = x \otimes (g,0) \]
In fact, the fusion operator is precisely $\mathsf{h}^\circ_G = \mathbb{Z}_2 \otimes_\mathbb{Z} \iota_1$. In Section \ref{sec:idempotent}, we will show that the fusion invertor is always of this form for any idempotent Hopf monad. 
\end{example}

\begin{example}\normalfont For any Abelian group $G$ and the Hopf monad $\mathsf{T}(M,N) := (G \oplus M, \mathsf{0})$ on $\mathsf{CMON} \times \mathsf{CMON}$, the fusion invertor is given by the pair $\mathsf{h}^\circ_{(M,N)} := (\mathsf{h}^\circ_M, 0)$, where $\mathsf{h}^\circ_M: G \oplus M \to G \oplus M \oplus G \oplus \mathsf{0}$ is the fusion invertor defined in Example \ref{ex:Gfi}. 
\end{example}

\section{Representable Hopf Monads}\label{sec:representable}

As discussed in the introduction, the main class of Hopf monads are those induced by Hopf monoids, which are called representable Hopf monads. In an arbitrary symmetric monoidal category, every (cocommutative) bimonoid $B$ induces a (symmetric) bimonad $\mathsf{T}(-) := B \otimes -$, while every (cocommutative) Hopf monoid $H$ induces a (symmetric) Hopf monad $\mathsf{T}(-) := H \otimes -$, where the inverse of the fusion operator is constructed using the antipode \cite[Example 2.10]{bruguieres2011hopf}. In this section, we consider representable Hopf monads in the biproduct case and show that the fusion invertor is always of the same form. 


In a category $\mathbb{X}$ with finite biproducts, recall that every object $H$ has a canonical and unique bimonoid structure, which is also bicommutative. The multiplication $H \oplus H \to H$ is the canonical codiagonal of the coproduct, the comultiplication $H \to H \oplus H$ of the product, while the unit $\mathsf{0} \to H$ and counit $H \to \mathsf{0}$ are the zero maps from and to the zero object. As such for every object $H$, we obtain a monad $\mathsf{T}(-):= H \oplus -$ on $\mathbb{X}$, where the monad multiplication $\mu_A: H \oplus A \oplus A \to H \oplus A$ and the monad unit $\eta_A: A \to H \oplus A$ are respectively defined as follows using the additive structure of the biproduct: 
\begin{align}\label{Hmonad}
    \mu_A := \iota_1 \circ \pi_1 +  \iota_1 \circ \pi_2 + \iota_2 \circ \pi_3 && \eta_A := \iota_2
\end{align}
Thus the induced fusion operator $\mathsf{h}_{A,B} : H \oplus A \oplus H \oplus B \to H \oplus A \oplus H \oplus B$ is given by: 
\begin{align}
   \mathsf{h}_{A,B} := \iota_1 \circ \pi_1 + \iota_2 \circ \pi_2 + \iota_3 \circ \pi_1 + \iota_3 \circ \pi_3 + \iota_4 \circ \pi_4 
\end{align}
In a category $\mathbb{X}$ with finite biproducts, an object $H$ is a Hopf monoid if and only if the identity $1_H: H \to H$ has an additive inverse $-1_H: H \to H$, that is, $1_H + (-1_H) = 0$. In this case, $-1_H$ is the antipode for the canonical bimonoid structure on $H$. Using $-1_H$ we may construct the inverse of the fusion operator $\mathsf{h}^{-1}_{A,B} : H \oplus A \oplus H \oplus B \to H \oplus A \oplus H \oplus B$ as follows: 
\begin{align}
   \mathsf{h}^{-1}_{A,B} := \iota_1 \circ \pi_1 + \iota_2 \circ \pi_2 + \iota_3 \circ (-1_H) \circ \pi_1 + \iota_3 \circ \pi_3 + \iota_4 \circ \pi_4 
\end{align}
Using ``element notation'', one can see that this is indeed a generalization of Example \ref{ex:Gx-}. 

\begin{lemma} Let $H$ be an object in a category $\mathbb{X}$ with finite biproducts such that its identity $1_H: H \to H$ has an additive inverse $-1_H: H \to H$. Then for the Hopf monad $H \oplus -$ on $\mathbb{X}$, the induced fusion invertor $\mathsf{h}^\circ_A: H \oplus A \to H \oplus A \oplus H \oplus \mathsf{0}$ is of the form: 
\begin{align} \label{H+hinvdef}
    \mathsf{h}^\circ_A = \iota_1 \circ \pi_1 + \iota_2 \circ \pi_1 + \iota_3 \circ (-1_H) \circ \pi_1 
\end{align}
\end{lemma}
\begin{proof} By Proposition \ref{mainprop}, applying (\ref{hinvtohcirc}) to the Hopf monad $H \oplus -$, we get that the fusion invertor $\mathsf{h}^\circ$ is given by the forumla $\mathsf{h}^\circ_A := \mathsf{h}^{-1}_{A,\mathsf{0}} \circ \left( \iota_1 \circ \pi_1 + \iota_2 \circ \pi_2 \right)$. Then expanding out the definition of $\mathsf{h}^{-1}$ and using the biproduct identities that $\pi_i \circ \iota_j = 0$ if $i \neq j$ and $\pi_i \circ \iota_i = 1$, we obtain precisely (\ref{H+hinvdef}). 
\end{proof}

Again, one can use ``element notation'' to see that this fusion invertor is indeed a generalization of Example \ref{ex:Gfi}.

\section{Idempotent Monads}\label{sec:idempotent}

In this section, we consider the case of Hopf monads which are also idempotent monads, since idempotent Hopf monads were of particular interest in \cite{hasegawa2022traced}. We will show that for an idempotent Hopf monad, the fusion invertor and inverse of the fusion operator are always of a specific form. In fact, we will show that checking that an idempotent monad is a Hopf monad amounts to checking one identity. 

Recall that an \textbf{idempotent monad} \cite[Proposition 4.2.3]{borceux1994handbook} on a category $\mathbb{X}$ is a monad $(T, \mu, \eta)$ on $\mathbb{X}$ such that the monad multiplication ${\mu_A: TT(A) \to T(A)}$ is a natural isomorphism, so $TT(A) \cong T(A)$ and $\mu^{-1}_A = \eta_{\mathsf{T}(A)} =  \mathsf{T}(\eta_A)$. Our first observation is that idempotent monads preserve zero maps and zero objects. 

\begin{lemma} \label{idemp0} Let $(\mathsf{T}, \mu, \eta)$ be a idempotent monad on a category $\mathbb{X}$ with finite biproducts. Then $\mathsf{T}$ preserves zero maps, that is, $\mathsf{T}(0) = 0$. Therefore $\mathsf{T}$ also preserves the zero object. 
\end{lemma}
\begin{proof} Suppose that $(\mathsf{T}, \mu, \eta)$ is an idempotent monad Then we compute: 
\begin{align*}
    \mathsf{T}(0) &=~ \mathsf{T}(0) \circ \mu_A \circ \eta_{\mathsf{T}(A)} \tag{$\mu \circ \eta =1$} \\
    &=~ \mathsf{T}\left( 0 \circ \eta_A \right)  \circ \mu_A \circ \eta_{\mathsf{T}(A)}  \\
    &=~ \mathsf{T}(0) \circ \mathsf{T}(\eta_A)  \circ \mu_A \circ \eta_{\mathsf{T}(A)} \\
    &=~ \mathsf{T}(0) \circ \eta_{\mathsf{T}(A)} \tag{Idem. monad so $\mathsf{T}(\eta_A) \circ \mu = 1$} \\
    &=~ \eta_{B} \circ 0 \tag{Nat. of $\eta$} \\
    &=~ 0 
\end{align*}
So $\mathsf{T}(0) = 0$. 
\end{proof}

The inverse of the above lemma is not necessarily true: some monads preserve zero maps but are not idempotent monads. However, for Hopf monads, preserving zero maps is equivalent to being idempotent. Using this fact, we can show that the fusion invertor $\mathsf{h}^\circ_A: \mathsf{T}(A) \to \mathsf{T}\left( A \oplus \mathsf{T}(\mathsf{0}) \right)$ must always be $\mathsf{T}(\iota_1): \mathsf{T}(A) \to \mathsf{T}\left( A \oplus \mathsf{T}(\mathsf{0}) \right)$. In fact, we can also show that if the fusion operator is of this form then the Hopf monad must also be an idempotent monad.  

\begin{lemma} \label{hcidemp} Let $(\mathsf{T}, \mu, \eta)$ be a Hopf monad on a category $\mathbb{X}$ with finite biproducts. Then the following are equivalent: 
 \begin{enumerate}[{\em (i)}]
 \item $(\mathsf{T}, \mu, \eta)$ is an idempotent monad;
 \item $\mathsf{T}$ preserves zero maps, that is, $\mathsf{T}(0) = 0$;
\item The fusion invertor $\mathsf{h}^\circ_A: \mathsf{T}(A) \to \mathsf{T}\left( A \oplus \mathsf{T}(\mathsf{0}) \right)$ is of the form:
\begin{align}\label{hinvti}
    \mathsf{h}^\circ_A = \mathsf{T}(\iota_1)
\end{align}
 \end{enumerate}
\end{lemma}
\begin{proof} Observe that $(i) \Rightarrow (ii)$ is simply Lemma \ref{idemp0}. For $(ii) \Rightarrow (iii)$, suppose that $\mathsf{T}$ preserves zero maps. Using \textbf{[FI.3.0]}, we compute that: 
\begin{align*}
\mathsf{T}(\iota_1) &=~ \left(  \mathsf{h}^\circ_A \circ \mathsf{T}(\pi_1) + \mathsf{T}(\iota_2) \circ \eta_{\mathsf{T}(\mathsf{0})} \circ \mu_\mathsf{0} \circ \mathsf{T}(\pi_2) \right) \circ \mathsf{T}(\iota_1)\tag{\textbf{[FI.3.0]}} \\
&=~ \mathsf{h}^\circ_A \circ \mathsf{T}(\pi_1)  \circ \mathsf{T}(\iota_1) + \mathsf{T}(\iota_2) \circ \eta_{\mathsf{T}(\mathsf{0})} \circ \mu_\mathsf{0} \circ \mathsf{T}(\pi_2)  \circ \mathsf{T}(\iota_1) \\
&=~ \mathsf{h}^\circ_A +  \mathsf{T}(\iota_2) \circ \eta_{\mathsf{T}(B)} \circ \mu_B \circ \mathsf{T}(0) \tag{$\pi_2 \circ \iota_1 =0$} \\
&=~ \mathsf{h}^\circ_A + \mathsf{T}(\iota_2) \circ \eta_{\mathsf{T}(B)} \circ \mu_B \circ 0 \tag{By assump. $\mathsf{T}(0) = 0$}\\
&=~ \mathsf{h}^\circ_A + 0 \\
&=~  \mathsf{h}^\circ_A
\end{align*}
So $\mathsf{h}^\circ_A = \mathsf{T}(\iota_1)$. For $(iii) \Rightarrow (i)$, suppose that $\mathsf{h}^\circ_A = \mathsf{T}(\iota_1)$. To prove that $(\mathsf{T}, \mu, \eta)$ is an idempotent monad, it suffices to show that $\eta_{\mathsf{T}(B)} \circ \mu_B = 1_{\mathsf{T}\mathsf{T}(B)}$ \cite[Proposition 4.2.3]{borceux1994handbook}. First using  \textbf{[FI.2]}, we compute that following: 
\begin{align*}
    0 &=~ \mu_B \circ \mathsf{T}(\pi_2) \circ \mathsf{h}^\circ_A \tag{\textbf{[FI.2]}} \\
    &=~ \mu_B \circ \mathsf{T}(\pi_2) \circ \mathsf{T}(\iota_1) \tag{By assump. $\mathsf{h}^\circ_A = \mathsf{T}(\iota_1)$} \\
    &=~ \mu_B \circ \mathsf{T}(0) \tag{$\pi_2 \circ \iota_1 = 0$}\\
    &=~ \mu_B \circ \mathsf{T}\left( \eta_B \circ 0 \right) \\
    &=~ \mu_B \circ \mathsf{T}\left( \eta_B \right) \circ \mathsf{T}(0) \\
    &=~ \mathsf{T}(0) \tag{$\mu \circ \mathsf{T}(\eta) = 1$}
\end{align*}
So $\mathsf{T}(0) = 0$. Then using this and \textbf{[FI.3]}, we compute that: 
\begin{align*}
\mathsf{T}(\iota_2) &=~ \left( \mathsf{T}\left(1_A \oplus \mathsf{T}(0) \right) \circ \mathsf{h}^\circ_A \circ \mathsf{T}(\pi_1) + \mathsf{T}(\iota_2) \circ \eta_{\mathsf{T}(B)} \circ \mu_B \circ \mathsf{T}(\pi_2) \right) \circ \mathsf{T}(\iota_2)\tag{\textbf{[FI.3]}} \\
&=~ \left( \mathsf{T}\left(1_A \oplus \mathsf{T}(0) \right) \circ \mathsf{T}(\iota_1) \circ \mathsf{T}(\pi_1) + \mathsf{T}(\iota_2) \circ \eta_{\mathsf{T}(B)} \circ \mu_B \circ \mathsf{T}(\pi_2) \right) \circ \mathsf{T}(\iota_2)\tag{By assump. $\mathsf{h}^\circ_A = \mathsf{T}(\iota_1)$} \\
&=~ \left(  \mathsf{T}(\iota_1) \circ \mathsf{T}(\pi_1) + \mathsf{T}(\iota_2) \circ \eta_{\mathsf{T}(B)} \circ \mu_B \circ \mathsf{T}(\pi_2) \right) \circ \mathsf{T}(\iota_2)\tag{Nat. of $\iota_1$} \\
&=~ \mathsf{T}(\iota_1) \circ \mathsf{T}(\pi_1)  \circ \mathsf{T}(\iota_2) + \mathsf{T}(\iota_2) \circ \eta_{\mathsf{T}(B)} \circ \mu_B \circ \mathsf{T}(\pi_2)  \circ \mathsf{T}(\iota_2) \\
&=~ \mathsf{T}(\iota_1) \circ \mathsf{T}(0) +  \mathsf{T}(\iota_2) \circ \eta_{\mathsf{T}(B)} \circ \mu_B \tag{$\pi_1 \circ \iota_2 =0$} \\
&=~ \mathsf{T}(\iota_1) \circ 0 +  \mathsf{T}(\iota_2) \circ \eta_{\mathsf{T}(B)} \circ \mu_B \tag{$\mathsf{T}(0) = 0$} \\
&=~ 0 +  \mathsf{T}(\iota_2) \circ \eta_{\mathsf{T}(B)} \circ \mu_B \\
&=~ \mathsf{T}(\iota_2) \circ \eta_{\mathsf{T}(B)} \circ \mu_B 
\end{align*}
So we have that $\mathsf{T}(\iota_2) = \mathsf{T}(\iota_2) \circ \eta_{\mathsf{T}(B)} \circ \mu_B$. Post-composing by $\mathsf{T}(\pi_2)$, since $\pi_2 \circ \iota_2 = 1$, we obtain $\eta_{\mathsf{T}(B)} \circ \mu_B = 1_{\mathsf{T}\mathsf{T}(B)}$. Therefore we can conclude that $(\mathsf{T}, \mu, \eta)$ is an idempotent monad. So we conclude that $(i) \Leftrightarrow (ii) \Leftrightarrow (iii)$. 
\end{proof}

It follows that for idempotent Hopf monads, the inverse of the fusion operator is always of the same form as well. 

\begin{corollary} Let $(\mathsf{T}, \mu, \eta)$ be an idempotent Hopf monad on a category $\mathbb{X}$ with finite biproducts. Then the inverse of the fusion operator $\mathsf{h}^{-1}_{A,B}: \mathsf{T}(A) \oplus \mathsf{T}(B) \to \mathsf{T}\left( A \oplus \mathsf{T}(B) \right)$ is of the form:
\begin{align}\label{hidemhinv}
  \mathsf{h}^{-1}_{A,B} = \mathsf{T}(\iota_1) \circ \pi_2 + \mathsf{T}(\iota_2) \circ \eta_{\mathsf{T}(B)} \circ \pi_2
\end{align}
\end{corollary}
\begin{proof} By Lemma \ref{hcidemp}, we have that $\mathsf{T}(0) = 0$. This means that for the zero object, we have that $\mathsf{T}\left(1_A \oplus \mathsf{T}(0) \right) = 1_{\mathsf{T}\left( A \oplus \mathsf{T}(\mathsf{0}) \right)}$. By Lemma \ref{hcidemp}, we also have that $\mathsf{h}^\circ_A = \mathsf{T}(\iota_1)$. So by Proposition \ref{mainprop}, applying the construction of (\ref{hcirctohinv}) and rewriting with the previous identities in mind, we obtain precisely (\ref{hidemhinv}). 
\end{proof}

Of course, for any monad $(\mathsf{T}, \mu, \eta)$ on any category with finite biproducts, the map $\mathsf{T}(\iota_1): \mathsf{T}(A) \to \mathsf{T}\left( A \oplus \mathsf{T}(\mathsf{0}) \right)$ can always be defined. However, while $\mathsf{T}(\iota_1)$ always satisfies \textbf{[FI.1]} and \textbf{[FI.2]}, $\mathsf{T}(\iota_1)$ will not in general also satisfy \textbf{[FI.3]}. So checking that an idempotent monad is a Hopf monad amounts to checking that $\mathsf{T}(\iota_1)$ satisfies \textbf{[FI.3]}. It turns out that for an idempotent monad, \textbf{[FI.3]} can be simplified even further. Therefore, checking that an idempotent monad is a Hopf monad is reduced to checking one identity. 

\begin{proposition} Let $(\mathsf{T}, \mu, \eta)$ be an idempotent monad on a category $\mathbb{X}$ with finite biproducts. Then $(\mathsf{T}, \mu, \eta)$ is a Hopf monad if and only if for all pairs of objects $A, B \in \mathbb{X}$ the following equality holds:
\begin{align} \label{hopfidemp}
    \mathsf{T}\left( 1_A \oplus 0 \right) + \mathsf{T}(0 \oplus 1_{\mathsf{T}(B)})= 1_{\mathsf{T}\left(A \oplus \mathsf{T}(B) \right) }
\end{align}
\end{proposition}
\begin{proof} Suppose that  $(\mathsf{T}, \mu, \eta)$ is a Hopf monad. By Lemma \ref{hcidemp}, we have that the fusion invertor is of the form $\mathsf{h}^\circ_A = \mathsf{T}(\iota_1)$. Using \textbf{[FI.3]}, we compute: 
\begin{align*}
  \mathsf{T}\left( 1_A \oplus 0 \right) + \mathsf{T}(0 \oplus 1_B) &=~ \mathsf{T}(\iota_1) \circ \mathsf{T}(\pi_1) +  \mathsf{T}(\iota_2) \circ \mathsf{T}(\pi_2) \tag{$\iota_1 \circ \pi_1 = 1 \oplus 0$ and $\iota_2 \circ \pi_2 = 0 \oplus 1$} \\
&=~ \mathsf{T}\left(1_A \oplus \mathsf{T}(0) \right) \circ \mathsf{T}(\iota_1) \circ \mathsf{T}(\pi_1) +  \mathsf{T}(\iota_2) \circ \mathsf{T}(\pi_2) \tag{Nat. of $\iota_1$} \\
&=~  \mathsf{T}\left(1_A \oplus \mathsf{T}(0) \right) \circ \mathsf{h}^\circ_A \circ \mathsf{T}(\pi_1) +  \mathsf{T}(\iota_2) \circ \mathsf{T}(\pi_2) \tag{\ref{hinvti}} \\
  &=~  \mathsf{T}\left(1_A \oplus \mathsf{T}(0) \right) \circ \mathsf{h}^\circ_A \circ \mathsf{T}(\pi_1) +  \mathsf{T}(\iota_2) \circ \eta_{\mathsf{T}(B)} \circ \mu_B \circ  \mathsf{T}(\pi_2) \tag{Idem. monad so $\eta \circ \mu = 1$} \\
  &=~  1_{\mathsf{T}\left( A \oplus \mathsf{T}(B) \right)} \tag{\textbf{[FI.3]}}
\end{align*}
So the desired equality (\ref{hopfidemp}) holds. 

Conversely, suppose that (\ref{hopfidemp}) holds. Define the natural transformation $\mathsf{h}^\circ_A: \mathsf{T}(A) \to \mathsf{T}\left( A \oplus \mathsf{T}(\mathsf{0}) \right)$ as $\mathsf{h}^\circ_A = \mathsf{T}(\iota_1)$. We will now show that $\mathsf{h}^\circ_A$ satisfies the three identities \textbf{[FI.1]}, \textbf{[FI.2]}, and \textbf{[FI.3]}. So we compute: 
    \begin{enumerate}[{\em (i)}] \item $\mathsf{T}(\pi_1) \circ \mathsf{h}^\circ_A = 1_{\mathsf{T}(A)}$
\begin{align*}
\mathsf{T}(\pi_1) \circ \mathsf{h}^\circ_A &=~ \mathsf{T}(\pi_1) \circ \mathsf{T}(\iota_1) \tag{Def. of $\mathsf{h}^\circ$} \\
&=~  1_{\mathsf{T}(A)} \tag{$\pi_1 \circ \iota_1 = 1$}
\end{align*}
\item $\mu_B \circ \mathsf{T}(\pi_2) \circ \mathsf{h}^\circ_A = 0$
\begin{align*}
    \mu_B \circ \mathsf{T}(\pi_2) \circ \mathsf{h}^\circ_A &=~ \mu_B \circ \mathsf{T}(\pi_2) \circ \mathsf{T}(\iota_1)  \tag{Def. of $\mathsf{h}^\circ$} \\
    &=~ \mu_B \circ \mathsf{T}(0) \tag{$\pi_2 \circ \iota_1 =0$}\\
    &=~ \mu_B \circ 0 \tag{Lem. \ref{idemp0}} \\ 
    &=~ 0 
\end{align*}
\item $\mathsf{T}\left(1_A \oplus \mathsf{T}(0) \right) \circ \mathsf{h}^\circ_A \circ \mathsf{T}(\pi_1) +  \mathsf{T}(\iota_2) \circ \eta_{\mathsf{T}(B)} \circ \mu_B \circ  \mathsf{T}(\pi_2)  = 1_{\mathsf{T}\left( A \oplus \mathsf{T}(B) \right)}$
\begin{align*}
&\mathsf{T}\left(1_A \oplus \mathsf{T}(0) \right) \circ \mathsf{h}^\circ_A \circ \mathsf{T}(\pi_1) +  \mathsf{T}(\iota_2) \circ \eta_{\mathsf{T}(B)} \circ \mu_B \circ  \mathsf{T}(\pi_2) \\
&=~ \mathsf{T}\left(1_A \oplus \mathsf{T}(0) \right) \circ \mathsf{h}^\circ_A \circ \mathsf{T}(\pi_1) +  \mathsf{T}(\iota_2) \circ  \mathsf{T}(\pi_2) \tag{Idem. monad so $\eta \circ \mu = 1$} \\
&=~ \mathsf{T}\left(1_A \oplus \mathsf{T}(0) \right) \circ \mathsf{T}(\iota_1) \circ \mathsf{T}(\pi_1) + \mathsf{T}(\iota_2) \circ  \mathsf{T}(\pi_2) \tag{Def. of $\mathsf{h}^\circ$}  \\
&=~ \mathsf{T}(\iota_1) \circ \mathsf{T}(\pi_1) + \mathsf{T}(\iota_2) \circ \mathsf{T}(\pi_2) \tag{Nat. of $\iota_1$} \\
&=~  \mathsf{T}\left( 1_A \oplus 0 \right) + \mathsf{T}(0 \oplus 1_B) \tag{$\iota_1 \circ \pi_1 = 1 \oplus 0$ and $\iota_2 \circ \pi_2 = 0 \oplus 1$} \\
&=~ 1_{\mathsf{T}\left( A \oplus \mathsf{T}(B) \right)} \tag{\ref{hopfidemp}}
\end{align*}
\end{enumerate}
Therefore, $\mathsf{h}^\circ$ is a fusion invertor for $(\mathsf{T}, \mu, \eta)$. Then by Proposition \ref{mainprop}, $(\mathsf{T}, \mu, \eta)$ is a Hopf monad. 
\end{proof}

\section{Negatives}

In this section, we consider Hopf monads on a category that has finite biproducts and also additive negatives (also sometimes called an additive category). We will show that in a setting with negatives, for any Hopf monad, its fusion invertor and inverse of the fusion operator are always of the same form. We will then show that in the presence of negatives, checking that a monad is a Hopf monad is simplified to checking one identity. 

In a category $\mathbb{X}$ with finite biproducts that also has negatives, every homset is an Abelian group, which means that every map $f: A \to B$ has an additive inverse $-f: A \to B$, that is, $f + (-f) = 0$. First observe that this implies that every object is a Hopf monoid, which means every object induces a Hopf monad. 

\begin{lemma}In a category $\mathbb{X}$ with finite biproducts that also has negatives, for every object $H$, the monad $\mathsf{T}(-) := H \oplus -$, as defined in Section \ref{sec:representable}, is a Hopf monad. 
\end{lemma}

Turning our attention from representable Hopf monads to arbitrary Hopf monads, we will now explain how using negatives, one can show that fusion invertor is always of a specific form. 

\begin{lemma}\label{hcircneg} Let $(\mathsf{T}, \mu, \eta)$ be a Hopf monad on a category $\mathbb{X}$ with finite biproducts that also has negatives. Then the fusion invertor $\mathsf{h}^\circ_A: \mathsf{T}(A) \to \mathsf{T}\left( A \oplus \mathsf{T}(\mathsf{0}) \right)$ is of the following form: 
\begin{align}\label{hneghcirc}
    \mathsf{h}^\circ_A = \mathsf{T}(\iota_1) - \mathsf{T}(\iota_2) \circ \eta_{\mathsf{T}(\mathsf{0})} \circ \mathsf{T}(0) 
\end{align}
and the inverse of the fusion operator $\mathsf{h}^{-1}_{A,B}: \mathsf{T}(A) \oplus \mathsf{T}(B) \to \mathsf{T}\left( A \oplus \mathsf{T}(B) \right)$ is of the form:
\begin{align}\label{hneghinv}
  \mathsf{h}^{-1}_{A,B} = \mathsf{T}(\iota_1) \circ \pi_1  -  \mathsf{T}(\iota_2) \circ \eta_{\mathsf{T}(B)} \circ \mathsf{T}(0) \circ \pi_1 + \mathsf{T}(\iota_2) \circ \eta_{\mathsf{T}(B)} \circ \pi_2
\end{align}
\end{lemma}
\begin{proof} Using \textbf{[FI.3.0]}, we first compute that: 
\begin{align*}
\mathsf{T}(\iota_1) &=~ \left( \mathsf{h}^\circ_A \circ \mathsf{T}(\pi_1) + \mathsf{T}(\iota_2) \circ \eta_{\mathsf{T}(\mathsf{0})} \circ \mu_\mathsf{0} \circ \mathsf{T}(\pi_2) \right) \circ \mathsf{T}(\iota_1)\tag{\textbf{[FI.3]}} \\
&=~ \mathsf{h}^\circ_A \circ \mathsf{T}(\pi_1)  \circ \mathsf{T}(\iota_1) + \mathsf{T}(\iota_2) \circ \eta_{\mathsf{T}(\mathsf{0})} \circ \mu_\mathsf{0} \circ \mathsf{T}(\pi_2)  \circ \mathsf{T}(\iota_1) \\
&=~ \mathsf{h}^\circ_A +  \mathsf{T}(\iota_2) \circ \eta_{\mathsf{T}(\mathsf{0})} \circ \mu_\mathsf{0} \circ \mathsf{T}(0) \tag{$\pi_1 \circ \iota_1 = 1$ and $\pi_2 \circ \iota_1 = 0$}\\
&=~ \mathsf{h}^\circ_A +  \mathsf{T}(\iota_2) \circ \eta_{\mathsf{T}(\mathsf{0})} \circ \mu_\mathsf{0} \circ \mathsf{T}\left( \eta_\mathsf{0} \circ 0 \right) \\
&=~ \mathsf{h}^\circ_A +  \mathsf{T}(\iota_2) \circ \eta_{\mathsf{T}(\mathsf{0})} \circ \mu_\mathsf{0} \circ \mathsf{T}\left( \eta_\mathsf{0} \right) \circ  \mathsf{T}(0) \\
&=~ \mathsf{h}^\circ_A +  \mathsf{T}(\iota_2) \circ \eta_{\mathsf{T}(\mathsf{0})} \circ \mathsf{T}(0) \tag{$\mu \circ \mathsf{T}(\eta) = 1$}
\end{align*} 
So we have that $\mathsf{h}^\circ_A +  \mathsf{T}(\iota_2) \circ \eta_{\mathsf{T}(\mathsf{0})} \circ \mathsf{T}(0) =\mathsf{T}(\iota_1)$. By subtracting $\mathsf{T}(\iota_2) \circ \eta_{\mathsf{T}(B)} \circ \mathsf{T}(0)$ from both sides, we finally obtain that $\mathsf{h}^\circ_A = \mathsf{T}(\iota_1) - \mathsf{T}(\iota_2) \circ \eta_{\mathsf{T}(B)} \circ \mathsf{T}(0)$ as desired. So by Proposition \ref{mainprop}, applying the construction of (\ref{hcirctohinv}), we compute:
\begin{align*}
    \mathsf{h}^{-1}_{A,B} &=~ \mathsf{T}\left(1_A \oplus \mathsf{T}(0) \right) \circ \mathsf{h}^\circ_A \circ \pi_1 + \mathsf{T}(\iota_2) \circ \eta_{\mathsf{T}(B)} \circ \pi_2 \tag{\ref{hcirctohinv}} \\
    &=~ \mathsf{T}\left(1_A \oplus \mathsf{T}(0) \right) \circ \left(  \mathsf{T}(\iota_1) - \mathsf{T}(\iota_2) \circ \eta_{\mathsf{T}(B)} \circ \mathsf{T}(0) \right) \circ \pi_1 + \mathsf{T}(\iota_2) \circ \eta_{\mathsf{T}(B)} \circ \pi_2 \tag{\ref{hneghcirc}} \\
    &=~ \mathsf{T}\left(1_A \oplus \mathsf{T}(0) \right) \circ  \mathsf{T}(\iota_1) \circ \pi_1  - \mathsf{T}\left(1_A \oplus \mathsf{T}(0) \right) \circ \mathsf{T}(\iota_2) \circ \eta_{\mathsf{T}(B)} \circ \mathsf{T}(0) \circ \pi_1 + \mathsf{T}(\iota_2) \circ \eta_{\mathsf{T}(B)} \circ \pi_2 \\
    &=~ \mathsf{T}(\iota_1) \circ \pi_1  -  \mathsf{T}(\iota_2) \circ \eta_{\mathsf{T}(B)} \circ \mathsf{T}(0) \circ \mathsf{T}(0) \circ \pi_1 + \mathsf{T}(\iota_2) \circ \eta_{\mathsf{T}(B)} \circ \pi_2  \tag{Nat. of $\iota_1$, $\iota_2$, and $\eta$} \\
    &=~  \mathsf{T}(\iota_1) \circ \pi_1  -  \mathsf{T}(\iota_2) \circ \eta_{\mathsf{T}(B)} \circ \mathsf{T}(0) \circ \pi_1 + \mathsf{T}(\iota_2) \circ \eta_{\mathsf{T}(B)} \circ \pi_2 \tag{$0 \circ 0 = 0$}
\end{align*}
So $\mathsf{h}^{-1}_{A,B} = \mathsf{T}(\iota_1) \circ \pi_1  -  \mathsf{T}(\iota_2) \circ \eta_{\mathsf{T}(B)} \circ \mathsf{T}(0) \circ \pi_1 + \mathsf{T}(\iota_2) \circ \eta_{\mathsf{T}(B)} \circ \pi_2$ as desired. 
\end{proof}

Observe that for any monad $(\mathsf{T}, \mu, \eta)$ on a category with biproducts that also has negatives, one can always define the map $\mathsf{T}(\iota_1) - \mathsf{T}(\iota_2) \circ \eta_{\mathsf{T}(\mathsf{0})} \circ \mathsf{T}(0)$. However, while said map will satisfy \textbf{[FI.1]} and \textbf{[FI.2]}, it may not satisfy \textbf{[FI.3]}. So in this setting, checking if a monad is a Hopf monad, one needs only check that said map satisfies \textbf{[FI.3]}. It turns out that the equality in question can be simplified even further. Therefore, in the presence of negatives, checking that a monad is a Hopf monad amounts to showing that one equality holds.  

\begin{proposition} Let $(\mathsf{T}, \mu, \eta)$ be a monad on a category $\mathbb{X}$ with finite biproducts that also has negatives. Then $(\mathsf{T}, \mu, \eta)$ is a Hopf monad if and only if for all pairs of objects $A, B \in \mathbb{X}$ the following equality holds:
\begin{align} \label{hopfneg}
    \mathsf{T}\left( 1_A \oplus 0 \right) + \mathsf{T}(\iota_2) \circ \eta_{\mathsf{T}(B)} \circ \mu_B \circ \mathsf{T}(\pi_2) - \mathsf{T}(\iota_2) \circ \eta_{\mathsf{T}(B)} \circ \mathsf{T}(0) = 1_{\mathsf{T}\left(A \oplus \mathsf{T}(B) \right) }
\end{align}
\end{proposition}
\begin{proof} Suppose that $(\mathsf{T}, \mu, \eta)$ is a Hopf monad. Using \textbf{[FI.3]}, we compute: 
\begin{align*}
&1_{\mathsf{T}\left( A \oplus \mathsf{T}(B) \right)} =~ \mathsf{T}\left(1_A \oplus \mathsf{T}(0) \right) \circ \mathsf{h}^\circ_A \circ \mathsf{T}(\pi_1) + \mathsf{T}(\iota_2) \circ \eta_{\mathsf{T}(B)} \circ \mu_B \circ \mathsf{T}(\pi_2) \\
&=~\mathsf{T}\left(1_A \oplus \mathsf{T}(0) \right) \circ \left(  \mathsf{T}(\iota_1) - \mathsf{T}(\iota_2) \circ \eta_{\mathsf{T}(B)} \circ \mathsf{T}(0) \right) \circ \mathsf{T}(\pi_1) + \mathsf{T}(\iota_2) \circ \eta_{\mathsf{T}(B)} \circ \mu_B \circ \mathsf{T}(\pi_2) \tag{\ref{hneghcirc}} \\
&=~ \mathsf{T}\left(1_A \oplus \mathsf{T}(0) \right) \circ \mathsf{T}(\iota_1) \circ \mathsf{T}(\pi_1) - \mathsf{T}\left(1_A \oplus \mathsf{T}(0) \right) \circ \mathsf{T}(\iota_2) \circ \eta_{\mathsf{T}(B)} \circ \mathsf{T}(0) \circ \mathsf{T}(\pi_1) + \mathsf{T}(\iota_2) \circ \eta_{\mathsf{T}(B)} \circ \mu_B \circ \mathsf{T}(\pi_2) \\
&=~ \mathsf{T}(\iota_1) \circ \mathsf{T}(\pi_1) -  \mathsf{T}(\iota_2) \circ \eta_{\mathsf{T}(B)} \circ \mathsf{T}(\pi_1) + \mathsf{T}(\iota_2) \circ \eta_{\mathsf{T}(B)} \circ \mu_B \circ \mathsf{T}(\pi_2) \tag{Nat. of $\iota_1$, $\iota_2$, and $\eta$, and $0 \circ 0 = 0$} \\
&=~ \mathsf{T}\left( 1_A \oplus 0 \right) - \mathsf{T}(\iota_2) \circ \eta_{\mathsf{T}(B)} \circ \mathsf{T}(0) + \mathsf{T}(\iota_2) \circ \eta_{\mathsf{T}(B)} \circ \mu_B \circ \mathsf{T}(\pi_2) \tag{$\iota_\mathsf{0} \circ \pi_\mathsf{0} = 1_A \oplus 0$} 
\end{align*}
So the desired equality (\ref{hopfneg}) holds. 

Conversely, suppose that (\ref{hopfneg}) holds. Define the natural transformation $\mathsf{h}^\circ_A: \mathsf{T}(A) \to \mathsf{T}\left( A \oplus \mathsf{T}(B) \right)$ as follows: 
\begin{align}
    \mathsf{h}^\circ_A = \mathsf{T}(\iota_1) - \mathsf{T}(\iota_2) \circ \eta_{\mathsf{T}(B)} \circ \mathsf{T}(0) 
\end{align}
We will now show that $\mathsf{h}^\circ_A$ satisfies the three identities \textbf{[FI.1]}, \textbf{[FI.2]}, and \textbf{[FI.3]}. So we compute: 
    \begin{enumerate}[{\em (i)}] \item $\mathsf{T}(\pi_1) \circ \mathsf{h}^\circ_A = 1_{\mathsf{T}(A)}$
\begin{align*}
\mathsf{T}(\pi_1) \circ \mathsf{h}^\circ_A &=~ \mathsf{T}(\pi_1) \circ \left( \mathsf{T}(\iota_1) - \mathsf{T}(\iota_2) \circ \eta_{\mathsf{T}(B)} \circ \mathsf{T}(0) \right) \tag{Def. of $\mathsf{h}^\circ$} \\
&=~ \mathsf{T}(\pi_1) \circ \mathsf{T}(\iota_1) - \mathsf{T}(\pi_1) \circ \mathsf{T}(\iota_2) \circ \eta_{\mathsf{T}(B)} \circ \mathsf{T}(0) \\
&=~ 1_{\mathsf{T}(A)} - \mathsf{T}(0) \circ \eta_{\mathsf{T}(B)} \circ \mathsf{T}(0) \tag{$\pi_1 \circ \iota_1 = 1$ and $\pi_1 \circ \iota_2 =0$} \\
&=~ 1_{\mathsf{T}(A)} - \eta_{A} \circ 0 \circ \mathsf{T}(0) \tag{Nat. of $\eta$} \\
&=~ 1_{\mathsf{T}(A)} - 0 \\
&=~  1_{\mathsf{T}(A)} 
\end{align*}
\item $\mu_B \circ \mathsf{T}(\pi_2) \circ \mathsf{h}^\circ_A = 0$
\begin{align*}
    \mu_B \circ \mathsf{T}(\pi_2) \circ \mathsf{h}^\circ_A &=~ \mu_B \circ \mathsf{T}(\pi_2) \circ \left( \mathsf{T}(\iota_1) - \mathsf{T}(\iota_2) \circ \eta_{\mathsf{T}(B)} \circ \mathsf{T}(0) \right) \tag{Def. of $\mathsf{h}^\circ$} \\
    &=~ \mu_B \circ \mathsf{T}(\pi_2) \circ \mathsf{T}(\iota_1) - \mu_B \circ \mathsf{T}(\pi_2) \circ \mathsf{T}(\iota_2) \circ \eta_{\mathsf{T}(B)} \circ \mathsf{T}(0) \tag{$\pi_2 \circ \iota_1 = 0$ and $\pi_2 \circ \iota_2 =1$} \\ 
    &=~ \mu_B \circ \mathsf{T}(0) - \mu_B \circ \eta_{\mathsf{T}(B)} \circ \mathsf{T}(0) \\
    &=~ \mu_B \circ \mathsf{T}\left( \eta_B \circ 0 \right) - \mathsf{T}(0) \tag{$\mu \circ \eta =1$} \\
    &=~ \mu_B \circ \mathsf{T}\left( \eta_B \right)   \circ \mathsf{T}(0) - \mathsf{T}(0) \\
    &=~ \mathsf{T}(0) - \mathsf{T}(0) \tag{$\mu \circ \mathsf{T}(\eta) = 1$}\\
    &=~ 0 
\end{align*}
\item $\mathsf{T}\left(1_A \oplus \mathsf{T}(0) \right) \circ \mathsf{h}^\circ_A \circ \mathsf{T}(\pi_1) + \mathsf{T}(\iota_2) \circ \eta_{\mathsf{T}(B)} \circ \mu_B \circ \mathsf{T}(\pi_2) =  1_{\mathsf{T}\left( A \oplus \mathsf{T}(B) \right)}$
\begin{align*}
&\mathsf{T}\left(1_A \oplus \mathsf{T}(0) \right) \circ \mathsf{h}^\circ_A \circ \mathsf{T}(\pi_1) + \mathsf{T}(\iota_2) \circ \eta_{\mathsf{T}(B)} \circ \mu_B \circ \mathsf{T}(\pi_2) \\
&=~  \mathsf{T}\left(1_A \oplus \mathsf{T}(0) \right)\circ \left( \mathsf{T}(\iota_1) - \mathsf{T}(\iota_2) \circ \eta_{\mathsf{T}(B)} \circ \mathsf{T}(0) \right)  \circ \mathsf{T}(\pi_1) + \mathsf{T}(\iota_2) \circ \eta_{\mathsf{T}(B)} \circ \mu_B \circ \mathsf{T}(\pi_2) \tag{Def. of $\mathsf{h}^\circ$} \\
&=~ \mathsf{T}\left(1_A \oplus \mathsf{T}(0) \right) \circ \mathsf{T}(\iota_1) \circ \mathsf{T}(\pi_1) - \mathsf{T}\left(1_A \oplus \mathsf{T}(0) \right)\circ  \mathsf{T}(\iota_2) \circ \eta_{\mathsf{T}(B)} \circ \mathsf{T}(0)  \circ \mathsf{T}(\pi_1) + \mathsf{T}(\iota_2) \circ \eta_{\mathsf{T}(B)} \circ \mu_B \circ \mathsf{T}(\pi_2) \\
&=~ \mathsf{T}(\iota_1) \circ \mathsf{T}(\pi_1) - \mathsf{T}(\iota_2) \circ \eta_{\mathsf{T}(B)} \circ \mathsf{T}(0) + \mathsf{T}(\iota_2) \circ \eta_{\mathsf{T}(B)} \circ \mu_B \circ \mathsf{T}(\pi_2) \tag{Nat. of $\iota_1$, $\iota_2$, and $\eta$, and $0 \circ 0 = 0$} \\
&=~ \mathsf{T}\left( 1_A \oplus 0 \right) - \mathsf{T}(\iota_2) \circ \eta_{\mathsf{T}(B)} \circ \mathsf{T}(0) + \mathsf{T}(\iota_2) \circ \eta_{\mathsf{T}(B)} \circ \mu_B \circ \mathsf{T}(\pi_2) \tag{$\iota_1 \circ \pi_1 = 1 \oplus 0$} \\
&=~ 1_{\mathsf{T}\left( A \oplus \mathsf{T}(B) \right)} \tag{\ref{hopfneg}}
\end{align*}
\end{enumerate}
Therefore, $\mathsf{h}^\circ$ is a fusion invertor for $(\mathsf{T}, \mu, \eta)$. Then by Proposition \ref{mainprop}, $(\mathsf{T}, \mu, \eta)$ is a Hopf monad. 
\end{proof}

\bibliographystyle{plain}      
\bibliography{tracereferences}   

\begin{thebibliography}{1}

\bibitem{borceux1994handbook}
F.~Borceux.
\newblock {\em Handbook of Categorical Algebra: Volume 2, Categories and
  Structures}.
\newblock Cambridge University Press, 1994.

\bibitem{bruguieres2011hopf}
A.~Bruguieres, S.~Lack, and A.~Virelizier.
\newblock Hopf monads on monoidal categories.
\newblock {\em Advances in Mathematics}, 227(2):745--800, 2011.

\bibitem{bruguieres2007hopf}
A.~Brugui{\`e}res and A.~Virelizier.
\newblock Hopf monads.
\newblock {\em Advances in Mathematics}, 215(2):679--733, 2007.

\bibitem{gutierrez2015solid}
J.~J. Guti{\'e}rrez.
\newblock On solid and rigid monoids in monoidal categories.
\newblock {\em Applied Categorical Structures}, 23(4):575--589, 2015.

\bibitem{hasegawa2018linear}
M.~Hasegawa and J.-S.~P. Lemay.
\newblock Linear distributivity with negation, star-autonomy, and {H}opf
  monads.
\newblock {\em Theory and Applications of Categories}, 33(37):1145--1157, 2018.

\bibitem{hasegawa2022traced}
M.~Hasegawa and J.-S.~P. Lemay.
\newblock Traced monads and hopf monads.
\newblock {\em arXiv preprint arXiv:2208.06529}, 2022.

\bibitem{mccrudden2002opmonoidal}
P.~McCrudden.
\newblock Opmonoidal monads.
\newblock {\em Theory Appl. Categ}, 10(19):469--485, 2002.

\bibitem{moerdijk2002monads}
I.~Moerdijk.
\newblock Monads on tensor categories.
\newblock {\em Journal of Pure and Applied Algebra}, 168(2):189--208, 2002.

\bibitem{selinger2003order}
P.~Selinger.
\newblock Order-incompleteness and finite lambda reduction models.
\newblock {\em Theoretical Computer Science}, 309(1-3):43--63, 2003.

\end{thebibliography}

\end{document}